\documentclass[12pt]{article}
\usepackage{natbib,graphicx,setspace,lscape,longtable}
\usepackage{epsfig,graphicx,float}
\usepackage{amsmath,amsthm,amssymb,color}

\usepackage[nohead]{geometry}
\usepackage[bottom]{footmisc}
\usepackage{indentfirst}
\usepackage{endnotes}
\usepackage{verbatim}
\usepackage{threeparttable}
\RequirePackage[mathlines]{lineno}
\usepackage{multirow}
\usepackage[titletoc]{appendix}
\usepackage{enumerate}
\bibpunct{(}{)}{;}{a}{,}{,}
\floatstyle{ruled}
\newfloat{algorithm}{tbhp}{loa}
\floatname{algorithm}{Algorithm}

\newcommand{\bb}{\mathrm{\bf b}}
\newcommand{\bff}{\mathrm{\bf f}}
\newcommand{\bx}{\mathrm{\bf x}}

\newcommand{\bA}{\mathrm{\bf A}}
\newcommand{\ba}{\mathrm{\bf a}}
\newcommand{\bu}{\mathrm{\bf u}}
\newcommand{\bv}{\mathrm{\bf v}}

\newcommand{\bm}{\mathrm{\bf m}}
\newcommand{\bB}{\mathrm{\bf B}}
\newcommand{\bP}{\mathrm{\bf P}}
\newcommand{\bcc}{\mathrm{\bf c}}
\newcommand{\bZ}{\mathrm{\bf Z}}
\newcommand{\bg}{\mathrm{\bf g}}

\newcommand{\bF}{\mathrm{\bf F}}
\newcommand{\bG}{\mathrm{\bf G}}
\newcommand{\bz}{\mathrm{\bf z}}
\newcommand{\bH}{\mathrm{\bf H}}
\newcommand{\bI}{\mathrm{\bf I}}

\newcommand{\bM}{\mathrm{\bf M}}
\newcommand{\bV}{\mathrm{\bf V}}

\newcommand{\bU}{\mathrm{\bf U}}
\newcommand{\bX}{\mathrm{\bf X}}

\newcommand{\bphi}{\mbox{\boldmath $\phi$}}
\newcommand{\bvarphi}{\mbox{\boldmath $\varphi$}}
\newcommand{\bpsi}{\mbox{\boldmath $\psi$}}
\newcommand{\bPhi}{\mbox{\boldmath $\Phi$}}

\newcommand{\bzero}{\mathrm{\bf 0}}

\newcommand{\bxi}{\mbox{\boldmath $\xi$}}

\newcommand{\bLam}{\mbox{\boldmath $\Lambda$}}

\newcommand{\bgamma}{\mbox{\boldmath $\gamma$}}

\newcommand{\bSigma}{\mbox{\boldmath $\Sigma$}}
\newcommand{\bGamma}{\mbox{\boldmath $\Gamma$}}
\newcommand{\bOmega}{\mbox{\boldmath $\Omega$}}

\newcommand{\hB}{\widehat \bB}

\newcommand{\hX}{\widehat \bX}

\newcommand{\hb}{\widehat \bb}
\newcommand{\hm}{\widehat \bm}

\newcommand{\hF}{\widehat \bF}

\newcommand{\hf}{\widehat \bff}

\newcommand{\hSig}{\widehat\Sig}
\newcommand{\hbSig}{\widehat\bSigma}

\newcommand{\hlam}{\widehat\lambda}
\newcommand{\hLam}{\widehat \bLam}

\newcommand{\hxi}{\widehat\bxi}
\newcommand{\hphi}{\widehat\bphi}
\newcommand{\hpsi}{\widehat\bpsi}

\newcommand{\cov}{\mathrm{cov}}
\newcommand{\Sig}{\mathbf{\Sigma}}

\newcommand{\diag}{\mathrm{diag}}

\newcommand{\var}{\mathrm{var}}

\newcommand{\beq}{\begin{eqnarray*}}
\newcommand{\eeq}{\end{eqnarray*}}

\newcommand{\R}{\mathbb{R}}
\DeclareMathOperator*{\argmax}{arg\,max}

\numberwithin{equation}{section}
\theoremstyle{plain}
\newtheorem{thm}{Theorem}[section]

\newtheorem{lem}{Lemma}[section]
\newtheorem{cor}{Corollary}[section]
\newtheorem{prop}{Proposition}[section]
\newtheorem{assum}{Assumption}[section]
\theoremstyle{definition}

\newtheorem{remark}{Remark}[section]

\def\ben{\begin{equation*}}
\def\een{\end{equation*}}
\def\bea{\begin{eqnarray}}
\def\eea{\end{eqnarray}}
\def\bean{\begin{eqnarray*}}
\def\eean{\end{eqnarray*}}
\def\bep{\begin{prop}}
\def\eep{\end{prop}}
\def\bc{\begin{center}}
\def\ec{\end{center}}

\numberwithin{equation}{section}

\begin{document}

\title{Sufficient Forecasting Using Factor Models}

\author{Jianqing Fan$^\dag$, Lingzhou Xue$^\ddag$ and Jiawei Yao$^\dag$\\ $^\dag$Princeton University and $^\ddag$Pennsylvania State University}
\date{}
\maketitle{}

\pagestyle{plain}

\begin{abstract}
We consider forecasting a single time series when there is a large number of predictors and a possible nonlinear effect. The dimensionality was first reduced via a high-dimensional (approximate) factor model implemented by the principal component analysis. Using the extracted factors, we develop a novel forecasting method called the sufficient forecasting, which provides a set of sufficient predictive indices, inferred from high-dimensional predictors, to deliver additional predictive power. The projected principal component analysis will be employed to enhance the accuracy of inferred factors when a semi-parametric (approximate) factor model is assumed. Our method is also applicable to cross-sectional sufficient regression using extracted factors. The connection between the sufficient forecasting and the deep learning architecture is explicitly stated. The sufficient forecasting correctly estimates projection indices of the underlying factors even in the presence of a nonparametric forecasting function. The proposed method extends the sufficient dimension reduction to high-dimensional regimes by condensing the cross-sectional information through factor models. We derive asymptotic properties for the estimate of the central subspace spanned by these projection directions as well as the estimates of the sufficient predictive indices. We further show that the natural method of running multiple regression of target on estimated factors yields a linear estimate that actually falls into this central subspace. Our method and theory allow the number of predictors to be larger than the number of observations. We finally demonstrate that the sufficient forecasting improves upon the linear forecasting in both simulation studies and an empirical study of forecasting macroeconomic variables.
\end{abstract}


\onehalfspacing

\section{Introduction}
Regression and forecasting in a data-rich environment have been an important research topic in statistics, economics and finance. Typical examples include forecasts of a macroeconomic output using a large number of employment and production variables \citep{stock-watson-1989,bernanke-etal-2005}, forecasts of the values of market prices and dividends using cross-sectional asset returns \citep{sharpe-1964,lintner-1965}, and studies of association between clinical outcomes and high throughput genomics and genetic information such as microarray gene expressions. The predominant framework to harness the vast predictive information is via the \textit{factor model}, which proves effective in simultaneously modeling the commonality and cross-sectional dependence of the observed data. It is natural to think that the latent factors drive simultaneously the vast covariate information as well as the outcome.  This achieves the dimensionality reduction in the regression and predictive models.

Turning the curse of dimensionality into blessing, the latent factors can be accurately extracted from the vast predictive variables and hence they can be reliably used to build models for response variables. For the same reason, factor models have been widely employed in many applications, such as portfolio management \citep{fama-1992, carhart-1997,fama-2014,hou2015}, large-scale multiple testing \citep{leek-storey-2008,fan-2012}, high-dimensional covariance matrix estimation \citep{fan-etal-2008, fan-2013}, and forecasting using many predictors \citep{stock-watson-2002,stock-watson-2002b}. Leveraging on the recent developments on the factor models, in this paper, we are particularly interested in the problem of forecasting. Our techniques can also be applied to regression problems, resulting in sufficient regression.

With little knowledge of the relationship between the forecast target and the latent factors, most research focuses on a linear model and its refinement. 
Motivated by the classic principal component regression \citep{kendall-1957,hotelling-1957}, \cite{stock-watson-2002,stock-watson-2002b} employed a similar idea to forecast a single time series from a large number of predictors: first use the principal component analysis (PCA) to estimate the underlying common factors, followed by a linear regression of the target on the estimated factors. The key insight here is to condense information from many cross-sectional predictors into several predictive indices. \cite{boivin-ng-2006} pointed out that the relevance of factors is important to the forecasting power, and may lead to improved forecast. As an improvement to this procedure, \cite{bair-etal-2006} used correlation screening to remove irrelevant predictors before performing the PCA. In a similar fashion, \cite{bai-2008} applied boosting and employed thresholding rules to select ``targeted predictors'', and \cite{stock-watson-2012} used shrinkage methods to downweight unrelated principal components. \cite{kelly-pruitt-2014} took into account the covariance with the forecast target in linear forecasting, and they proposed a three-pass regression filter that generalizes partial least squares to forecast a single time series.

As mentioned, most of these approaches are fundamentally limited to linear forecasting. This yields only one index of extracted factors for forecasting.  It does not provide sufficient predictive power when the predicting function is nonlinear and depends on multiple indices of extracted factors. Yet, nonlinear models usually outperform linear models in time series analysis \citep{Tjostheim-1994}, especially in multi-step prediction. 
When the link function between the target and the factors is arbitrary and unknown, a thorough exploration of the factor space often leads to additional gains. To the best of our knowledge, there are only a few papers in forecasting a nonlinear time series using factor models. \cite{bai-2008} discussed the use of squared factors (i.e., volatility of the factors) in augmenting forecasting equation. \cite{Ludvigson-Ng-2007} found that the square of the first factor estimated from a set of financial factors is significant in the regression model for mean excess returns. This naturally leads to the question of which effective factor, or more precisely, which effective direction of factor space to include for higher moments. \cite{cunha-etal-2010} used a class of nonlinear factor models with shape constraints to study cognitive and noncognitive skill formation. The nonparametric forecasting function would pose a significant challenge in estimating effective predictive indices.

In this work, we shall address this issue from a completely different perspective to existing forecasting methods. We introduce a favorable alternative method called the \emph{sufficient forecasting}. Our proposed forecasting method springs from the idea of the sliced inverse regression, which was first introduced in the seminal work by \cite{li-1991}. In the presence of some  arbitrary and unknown (possibly nonlinear) forecasting function, we are interested in estimating a set of \textit{sufficient predictive indices}, given which the forecast target is independent of unobserved common factors. To put it another way, the forecast target relates to the unobserved common factors only through these sufficient predictive indices. Such a goal is closely related to the estimation of the \emph{central subspace} in the dimension reduction literature \citep{cook-2009}. In a linear forecasting model, such a central subspace consists of only one dimension. In contrast, when a nonlinear forecasting function is present, the central subspace can go beyond one dimension, and our proposed method can effectively reveal multiple sufficient predictive indices to enhance the prediction power. This procedure therefore greatly enlarges the scope of forecasting using factor models. As demonstrated in numerical studies, the sufficient forecasting has improved performance over benchmark methods, especially under a nonlinear forecasting equation.

In summary, the contribution of this work is at least twofold. On the one hand, our proposed sufficient forecasting advances existing forecasting methods, and fills the important gap between incorporating target information and dealing with nonlinear forecasting. Our work identifies effective factors that influence the forecast target without knowing the nonlinear dependence.
On the other hand, we present a promising dimension reduction technique through factor models. It is well-known that existing dimension reduction methods are limited to either a fixed dimension or a diverging dimension that is smaller than the sample size \citep{zhu-etal-2006}. With the aid of factor models, our work alleviates what plagues sufficient dimension reduction in high-dimensional regimes by condensing the cross-sectional
information, whose dimension can be much higher than the sample size.

The rest of this paper is organized as follows. Section 2 first proposes the sufficient forecasting using factor models, and then extends the proposed method by using semi-parametric factor models. In Section 3, we establish the asymptotic properties for the sufficient forecasting, and we also prove that the simple linear estimate actually falls into the central subspace. Section 4 demonstrates the numerical performance of the sufficient forecasting in simulation studies and an empirical application. Section 5 includes a few concluding remarks. Technical proofs are rendered in the Appendix.

\section{Sufficient forecasting}

This section first presents a unified framework for forecasting using factor models. We then propose the sufficient forecasting procedure with unobserved factors, which estimates predictive indices without requiring estimating the unknown forecasting function.

\subsection{Factor models and forecasting}
Consider the following factor model with a target variable $y_{t+1}$ which we wish to forecast:
\bea
y_{t+1}&=&h(\bphi_1'\bff_t,\cdots, \bphi_L'\bff_t,\epsilon_{t+1}),  \label{eq2.1}\\
x_{it}&=&\bb_i'\bff_t+u_{it},\quad 1\leq i\leq p,\ 1\leq t\leq T, \label{eq2.2}
\eea
where $x_{it}$ is the $i$-th predictor observed at time $t$, $\bb_i$ is a $K\times 1$ vector of factor loadings, $\bff_t=(f_{1t},\cdots, f_{Kt})'$ is a $K\times 1$ vector of common factors driving both the predictors and the response, and $u_{it}$ is the error term, or the idiosyncratic component. For ease of notation, we write $\bx_t=(x_{1t},\cdots,x_{pt})', \bB=(\bb_1,\cdots,\bb_p)'$ and $\bu_t=(u_{1t},\cdots,u_{pt})'$. In \eqref{eq2.1}, $h(\cdot)$ is an unknown link function, and $\epsilon_{t+1}$ is some stochastic error independent of $\bff_t$ and $u_{it}$. The vectors of linear combinations $\bphi_1,\cdots,\bphi_L$ are $K$-dimensional orthonormal vectors. Clearly, the model is also applicable to the cross-sectional regression such as those mentioned at the introduction of the paper. {Note that the sufficient forecasting can be represented in a deep learning architecture \citep{bengio-2009,bengio-etal-2013}, consisting of four layers of linear/nonlinear processes for dimension reduction and forecasting. The connection between sufficient forecasting and deep learning is illustrated in Figure 1. An advantage of our deep learning model is that it admits scalable and explicit computational algorithm.}

\begin{figure}[!htbp]
\begin{center}
\includegraphics[scale=0.4]{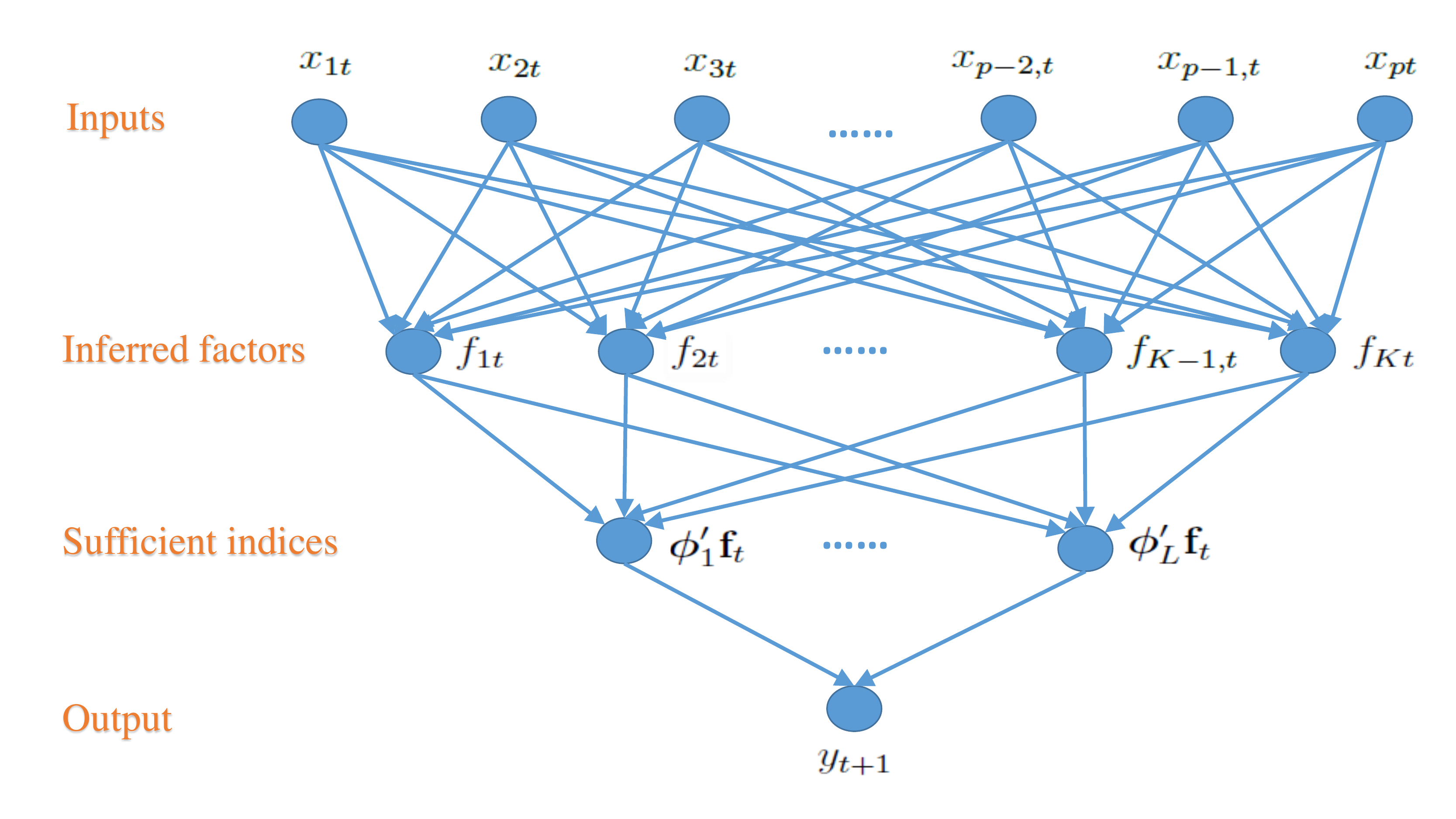}
\caption{\small The sufficient forecasting is represented as a four-layer deep learning architecture. The factors $f_{1t}, \cdots, f_{Kt}$ are obtained from the inputs via principal component analysis.}
\end{center}
\end{figure}

Model \eqref{eq2.1} is a semiparametric multi-index model with latent covariates and unknown nonparametric function $h(\cdot)$. The target $y_{t+1}$ depends on the factors $\bff_t$ only through $L$ predictive indices $\bphi_1'\bff_t,\cdots, \bphi_L'\bff_t$, where $L$ is no greater than $K$. In other words, these predictive indices are \textit{sufficient} in forecasting $y_{t+1}$.  Models \eqref{eq2.1} and \eqref{eq2.2} effectively reduce the dimension from the diverging $p$ to a fixed $L$ to estimate the nonparametric function $h(\cdot)$, and greatly alleviates the curse of dimensionality. Here, $\bphi_i$'s are also called the sufficient dimension reduction (SDR) directions, and their linear span forms the \textit{central subspace} \citep{cook-2009} denoted by $S_{y|\bff}$. Note that the individual directions $\bphi_1,\cdots,\bphi_L$ are not identifiable without imposing any structural condition on $h(\cdot)$. However, the subspace $S_{y|\bff}$ spanned by $\bphi_1,\cdots,\bphi_L$ can be identified. Therefore, throughout this paper, we refer any orthonormal basis $\bpsi_1,\cdots,\bpsi_L$ of the central subspace $S_{y|\bff}$ as {sufficient dimension reduction directions}, and their corresponding predictive indices $\bpsi_1'\bff_t,\cdots, \bpsi_L'\bff_t$ as {sufficient predictive indices}.

Suppose that the factor model (\ref{eq2.2}) has the following canonical normalization
\begin{equation}
\cov(\bff_t)=\bI_K \text{ and } \bB'\bB \text{ is diagonal}, \label{eq2.4}
\end{equation}
where $\bI_K$ is a $K\times K$ identity matrix. This normalization serves as an identification condition of $\bB$ and $\bff_t$ because $\bB\bff_t=(\bB\bOmega)(\bOmega^{-1}\bff_t)$ holds for any nonsingular matrix $\bOmega$. We assume for simplicity that the means of $x_{it}$'s and $\bff_t$'s in (\ref{eq2.2}) have already been removed, namely, $E(x_{it})=0$ and $E(\bff_t)=0$.

As a special case, suppose we have \textit{a priori} knowledge that the link function $h(\cdot)$ in (\ref{eq2.1}) is in fact linear.  In this case, no matter what $L$ is, there is only a single index of the factors $\bff_t$ in the prediction equation, i.e. $L=1$. Now, \eqref{eq2.1} reduces to
$$y_{t+1}=\bphi_1'\bff_t+\epsilon_{t+1}.$$
Such linear forecasting problems using many predictors have been studied extensively in the econometrics literature, for example, \cite{stock-watson-2002,stock-watson-2002b}, \cite{bai-2008}, and \cite{kelly-pruitt-2014}, among others. Existing methods are mainly based on the principal component regression (PCR), which is a multiple linear regression of the forecast target on the estimated principal components. The linear forecasting framework does not thoroughly explore the predictive power embedded in the underlying common factors.  It can only create a single index of the unobserved factors to forecast the target. 
The presence of a nonlinear forecasting function with multiple forecasting indices would only deteriorate the situation. While extensive efforts have been made such as \cite{boivin-ng-2006}, \cite{bai-2009}, and \cite{stock-watson-2012}, these approaches are more tailored to handle issues of using too many factors, and they are fundamentally limited to the linear forecasting.

\subsection{Sufficient forecasting}

Traditional analysis of factor models typically focuses on the covariance with the forecast target $\cov(\bx_t,y_{t+1})$ and the covariance within the predictors $\bx_t$, denoted by a $p\times p$ matrix
\begin{equation}
  \bSigma_x=\bB\cov(\bff_t)\bB'+\bSigma_u, \label{eq2.3}
\end{equation}
where $\bSigma_u$ is the error covariance matrix of $\bu_t$. Usually $\cov(\bx_t,y_{t+1})$ and $\bSigma_x$ are not enough to construct an optimal forecast, especially in the presence of a nonlinear forecasting function. To fully utilize the target information, our method considers the covariance matrix of the inverse regression curve, $E(\bx_t|y_{t+1})$. 
By conditioning on the target $y_{t+1}$ in model \eqref{eq2.2}, we obtain 
\begin{equation*}
\cov(E(\bx_t|y_{t+1}))=\bB\cov(E(\bff_t|y_{t+1}))\bB',
\end{equation*}
where we used the assumption that $E(\bu_t | y_{t+1}) = 0$.
This salient feature does not impose any structure on the covariance matrix of $\bu_t$. Under model (\ref{eq2.1}), \cite{li-1991} showed that $E(\bff_t|y_{t+1})$ is contained in the central subspace $S_{y|\bff}$ spanned by $\bphi_1, \ldots,\bphi_L$ provided that $E(\bb'\bff_t|\bphi_1'\bff_t,\cdots,\bphi_L'\bff_t)$ is a linear function of $\bphi_1'\bff_t,\cdots,\bphi_L'\bff_t$ for any $\bb\in\R^K$. This important result implies that $S_{y|\bff}$ contains the linear span of $\cov(E(\bff_t|y_{t+1}))$. Thus, it is promising to estimate sufficient directions by investigating the top $L$ eigenvectors of $\cov(E(\bff_t|y_{t+1}))$. To see this, let $\bPhi = (\bphi_1, \cdots, \bphi_L)$.  Then, we can write
\begin{equation} \label{eq2.5a}
   E(\bff_t|y_{t+1}) = \bPhi \ba(y_{t+1}),
\end{equation}
for some $L\times 1$ coefficient function $\ba(\cdot)$, according to the aforementioned result by \cite{li-1991}.  As a result,
\begin{equation} \label{eq2.5b}
\cov(E(\bff_t|y_{t+1})) =  \bPhi E [\ba(y_{t+1}) \ba(y_{t+1})^T] \bPhi'.
\end{equation}
This matrix has $L$ nonvanishing eigenvalues if $E [\ba(y_{t+1}) \ba(y_{t+1})^T]$ is non-degenerate.  Their corresponding eigenvectors have the same linear span as
$\bphi_1, \cdots, \bphi_L$.

It is difficult to directly estimate the covariance of the inverse regression curve $E(\bff_t|y_{t+1})$ and obtain sufficient predictive indices. Instead of using the conventional nonparametric regression method, we follow  \cite{li-1991} to explore the conditional information of these underlying factors. To this end, we introduce the sliced covariance estimate
\begin{equation}\label{eq2.5}
\bSigma_{f|y}=\frac{1}{H}\sum_{h=1}^HE(\bff_t|y_{t+1}\in I_h)E(\bff_t'|y_{t+1}\in I_h),
\end{equation}
where the range of $y_{t+1}$ is divided into $H$ slices $I_1,\ldots,I_H$ such that $P(y_{t+1}\in I_h) = 1/H$.  By (\ref{eq2.5a}), we have
\begin{equation} \label{eq2.5c}
\bSigma_{f|y}= \bPhi \left [ \frac{1}{H}\sum_{h=1}^HE(\ba(y_{t+1})|y_{t+1}\in I_h)
E(\ba(y_{t+1})|y_{t+1}\in I_h)^T \right ] \bPhi',
\end{equation}
This matrix has $L$ nonvanishing eigenvalues as long as the matrix in the bracket is non-degenerate.   In this case,
the linear span created by the eigenvectors of the $L$ largest eigenvalues of $\bSigma_{f|y}$
 is the same as that spanned by $\bphi_1, \cdots, \bphi_L$.  Note that $H\ge \max\{L,2\}$ is required in order for the matrix in the bracket to be non-degenerate.


To estimate $\bSigma_{f|y}$ in \eqref{eq2.5} with unobserved factors, a natural solution is to first consistently estimate factors $\bff_t$ from the factor model \eqref{eq2.2} and then use the estimated factors $\hf_t$ and the observed target $y_{t+1}$ to approximate the sliced estimate $\bSigma_{f|y}$. Alternatively, we can start with the observed predictors $\bx_t$. Since $E(\bu_t|y_{t+1}) = 0$, we then have
$$
E(\bx_t|y_{t+1}\in I_h) = \bB E(\bff_t|y_{t+1}\in I_h),
$$
and
$$
 E(\bff_t|y_{t+1}\in I_h) = \bLam_b E(\bx_t|y_{t+1}\in I_h),
$$
where $\bLam_b=(\bB'\bB)^{-1}\bB'$.
Hence, 
\begin{equation}\label{eq2.6}
\bSigma_{f|y}=\frac{1}{H}\sum_{h=1}^H \bLam_b E(\bx_t|y_{t+1}\in I_h)E(\bx_t'|y_{t+1}\in I_h) \bLam_b',
\end{equation}
Notice that $\bLam_b$ is also \textit{unknown} and needs to be estimated. With a consistent estimator $\hLam_b$, we can use $\bx_t$ and $y_{t+1}$ to approximate the sliced estimate $\bSigma_{f|y}$.

In the sequel, we present two estimators, $\hSig_{f|y}^1$ in \eqref{eq2.9} and $\hSig_{f|y}^2$ in \eqref{eq2.10}, estimated by using respectively \eqref{eq2.5} and \eqref{eq2.6}. Interestingly, we show that these two estimators are exactly equivalent when factors and factor loadings are estimated from the constrained least squares. We denote two equivalent estimators $\hSig_{f|y}^1$ and $\hSig_{f|y}^2$ by $\hSig_{f|y}$. As in Remark~\ref{rem2.1}, these two are in general not equivalent.

In Section 3, we shall show that under mild conditions, $\bSigma_{f|y}$ is consistently estimated by $\hbSig_{f|y}$ as $p,T\rightarrow \infty$. Furthermore, the eigenvectors of $\hbSig_{f|y}$ corresponding to the $L$ largest eigenvalues, denoted as $\hpsi_j$ $(j=1,\cdots,L)$, will converge to the corresponding eigenvectors of $\bSigma_{f|y}$, which actually span the aforementioned central subspace $S_{y|\bff}$. This will yield consistent estimates of sufficient predictive indices $\hpsi_1'\hf_t,\cdots,\hpsi_L'\hf_t$. Given these estimated low-dimensional indices, we may employ one of the well-developed nonparametric regression techniques to estimate $h(\cdot)$ and make the forecast, including local polynomial modeling, regression splines, additive modeling, and many others \citep{green-silverman-1993,fan-1996,li-racine-2007}. We summarize the sufficient forecasting procedure in Algorithm \ref{Algo2}.

\begin{algorithm}[H]
\caption{Sufficient forecasting using factor models}
\smallskip
\begin{itemize}
  \item Step 1: Obtain the estimated factors $\{\hf_t\}_{t=1,\ldots,T}$ from \eqref{eq2.7} and \eqref{eq2.8};
  \item Step 2: Construct $\hSig_{f|y}$ as in \eqref{eq2.9} or \eqref{eq2.10};
  \item Step 3:  Obtain $\hpsi_1,\cdots,\hpsi_L$ from the $L$ largest eigenvectors of $\hSig_{f|y}$;
  \item Step 4: Construct the predictive indices $\hpsi_1'\hf_t,\cdots,\hpsi_L'\hf_t$ and forecast $y_{t+1}$.
\end{itemize}
\label{Algo2}
\end{algorithm}

To make the forecasting procedure concrete, we elucidate how factors and factor loadings are estimated. We temporarily assume that the number of underlying factors $K$ is known to us. 
Consider the following constrained least squares problem:
\bea
(\hB_K,\hF_K) &=\text{arg}\min\limits_{(\bB,\bF)}&\|\bX-\bB\bF'\|^2_F,         \label{eq2.7}\\
&\text{subject to}& T^{-1}\bF'\bF=\bI_K, \quad \bB'\bB \text{ is diagonal,}  \label{eq2.8}
\eea
where $\bX=(\bx_1,\cdots,\bx_T)$, $\bF' =(\bff_1,\cdots,\bff_T)$, and $\|\cdot\|_F$ denotes the Frobenius norm. This is a classical principal components problem, and it has been widely used to extract underlying common factors \citep{stock-watson-2002,bai-2002,fan-2013,bai-ng-2013}. The constraints in (\ref{eq2.8}) correspond to the normalization (\ref{eq2.4}). The minimizers $\hF_K$ and $\hB_K$ are such that the columns of $\hF_K/\sqrt{T}$ are the eigenvectors corresponding to the $K$ largest eigenvalues of the $T\times T$ matrix $\bX'\bX$ and $$\hB_K=T^{-1}\bX\hF_K.$$ To simplify notation, we use $\hB=\hB_K$, $\hF=\hF_K$ and $\hF'=(\hf_1,\cdots,\hf_T)$ throughout this paper.

To effectively estimate $\bSigma_{f|y}$ in \eqref{eq2.5} and \eqref{eq2.6}, we shall replace the conditional expectations $E(\bff_t|y_{t+1}\in I_h)$ and $E(\bx_t|y_{t+1}\in I_h)$ by their sample counterparts. Denote the ordered statistics of $\{(y_{t+1},\hf_t)\}_{t=1,\ldots,T-1}$ by $\{(y_{(t+1)},\hf_{(t)})\}_{t=1,\ldots,T-1}$ according to the values of $y$, where $y_{(2)}\leq\cdots\leq y_{(T)}$ and we only use information up to time $T$. We divide the range of $y$ into $H$ slices, where $H$ is typically fixed. Each of the first $H-1$ slices contains the same number of observations $c>0$, and the last slice may have less than $c$ observations, which exerts little influence asymptotically. For ease of presentation, we introduce a double script $_{(h,j)}$ in which $h=1,\ldots,H$ refers to the slice number and $j=1,\ldots,c$ is the index of an observation in a given slice. Thus, we can write $\{(y_{(t+1)},\hf_{(t)})\}_{t=1,\ldots,T-1}$ as
$$
\{(y_{(h,j)},\hf_{(h,j)}): \ y_{(h,j)}=y_{(c(h-1)+j+1)}, \ \hf_{(h,j)}=\hf_{(c(h-1)+j)}\}_{h=1,\ldots,H; \ j=1,\ldots,c}.
$$

Based on the estimated factors $\hf_t$, we have the estimate $\hSig_{f|y}^1$ in the form of
\begin{equation}
\hSig_{f|y}^1=\frac{1}{H}\sum_{h=1}^H[\frac{1}{c}\sum_{l=1}^c\hf_{(h,l)}][\frac{1}{c}\sum_{l=1}^c\hf_{(h,l)}]'.  \label{eq2.9}
\end{equation}
Analogously, by using \eqref{eq2.6}, we have an alternative estimator:
\begin{equation}
\hSig_{f|y}^2=\hLam_b(\frac{1}{H}\sum_{h=1}^H[\frac{1}{c}\sum_{l=1}^c\bx_{(h,l)}][\frac{1}{c}\sum_{l=1}^c\bx_{(h,l)}]')\hLam_b'.  \label{eq2.10}
\end{equation}
where $\hLam_b=(\hB'\hB)^{-1}\hB$ is the plug-in estimate of $\bLam_b$.
Two estimates of $\bSigma_{f|y}$ are seemingly different: $\hSig_{f|y}^1$ based on estimated factors and $\hSig_{f|y}^2$ based on estimated loadings. When factors and loadings are estimated using \eqref{eq2.7} and \eqref{eq2.8}, the following proposition shows that $\hSig_{f|y}^1$ and $\hSig_{f|y}^2$ are the same.

\begin{prop} \label{prop2.1}
Let $\hf_t$ and $\hB$ be estimated by solving
the constrained least square problem in \eqref{eq2.7} and \eqref{eq2.8}. Then, the two estimators (\ref{eq2.9}) and (\ref{eq2.10}) of $\bSigma_{f|y}$ are equivalent, i.e.
$$\hSig_{f|y}^1 = \hSig_{f|y}^2.$$
\end{prop}

\smallskip
\begin{remark}
\label{rem2.1}
There are alternative ways for estimating factors $\bF$ and factor loadings $\bB$. For example, \cite{forni-2000,forni-2015} studied factor estimation based on projection, \cite{liu2012high}, \cite{xue2012regularized} and \cite{fxz-2015} employed rank-based methods for robust estimation of the covariance matrices, and \cite{connor-2012} applied a weighted additive nonparametric estimation procedure to estimate characteristic-based factor models.  These methods do not necessarily lead to the equivalence of $\hSig_{f|y}^1$ and $\hSig_{f|y}^2$ as in Proposition \ref{prop2.1}. However, the proposed sufficient forecasting will benefit from using generalized dynamic factor models \citep{forni-2000,forni-2015} or more robust factor models. This would go beyond the scope of this paper, and is left for future research.
\end{remark}

\subsection{Sufficient forecasting with projected principal components}

Often in econometrics and financial applications, there may exist explanatory variables to model the loading matrix of factor models \citep{connor-2012,fan-2015}. In such cases, the loading coefficient $\bb_i$ in \eqref{eq2.12} can be modeled by nonparametric functions $\bb_i=\bg(\bz_i)+\bgamma_i$, where $\bz_i$ denotes the vector of $d$ subject-specific covariates, and $\bgamma_i$ is the part of $\bb_i$ that can not be explained by the covariates $\bz_i$. Suppose that $\bgamma_i$'s are independently identically distributed with mean zero, and they are independent of $\bz_i$, $u_{it}$, and $\epsilon_{t+1}$.

Now we consider the following semi-parametric factor model to forecast $y_{t+1}$:
\bea
y_{t+1}&=&h(\bphi_1'\bff_t,\cdots, \bphi_L'\bff_t,\epsilon_{t+1}),  \label{eq2.11}\\
x_{it}&=&(\bg(\bz_i)+\bgamma_i)'\bff_t+u_{it},\quad 1\leq i\leq p,\ 1\leq t\leq T, \label{eq2.12}
\eea
where $\bphi_1,\ldots,\bphi_L$, $\epsilon_{t+1}$, $\bff_t$, and $u_{it}$ follow the same assumptions in Section 2.1. Note that \eqref{eq2.12} has the matrix form of $\bX=(\bG(\bZ)+\bGamma)\bF'+\bU$, which decomposes the factor loading matrix $\bB$ into the subject-specific component $\bG(\bZ)$ and the orthogonal residual component $\bGamma$. Following \cite{fan-2015}, the unobserved latent vectors $\bff_t$ in the semi-parametric factor model \eqref{eq2.12} can be  better estimated by an order of magnitude, as long as the contributions of $\bz_i$ are genuine.

Given covariates $\bz_i$, we first use the projected PCA \citep{fan-2015} to estimate latent factors. 
Projected PCA is implemented by projecting the predictors $\bX$ onto the sieve space spanned by $\{\bz_1,\ldots,\bz_p\}$ and applying the conventional PCA on the projected predictors $\hX$. As an example, we use additive spline model to approximate the function $g$ in (\ref{eq2.12}). Let $\varphi_1,\cdots,\varphi_J$ is a set of $J$ basis functions. Then, we define $\bvarphi(\bz_i)=(\varphi_1(z_{i1}),\ldots,\varphi_J(z_{i1}),\cdots,\varphi_1(z_{id}),\ldots,\varphi_J(z_{id}))'$, and $\bvarphi(\bZ)=(\bvarphi(\bz_1),\cdots,\bvarphi(\bz_p))'$ as a $p$ by $Jd$ design matrix, created by fitting additive spline models. Then, the projected predictors
(or smoothed predictors by using additive model) is
\begin{equation}\label{eq2.13}
\hX=\bP\bX,
\end{equation}
where
$
\bP=\bvarphi(\bZ) \left(\bvarphi(\bZ)'\bvarphi(\bZ)\right)^{-1}\bvarphi(\bZ)'\bX
$
is the projection matrix onto the sieve space. The projected PCA is to run PCA on the projected matrix $\hX$ to extract the latent factor $\hF$.
The advantage of this is that $\hF$ is estimated far more accurately, according to \cite{fan-2015}.  Note that the projection matrix $\bP$ was constructed by using the additive spline model.  It can also be constructed by using the multivariate kernel function.

After obtaining $\hF$, we simply follow Section 2.2 to obtain SDR directions $\hpsi_1,\cdots,\hpsi_L$ and sufficient predictive indices $\hpsi_1'\hf_t,\cdots,\hpsi_L'\hf_t$ to forecast forecast $y_{t+1}$.


\subsection{Choices of tuning parameters}

Sufficient forecasting includes determination of the number of slices $H$, the number of predictive indices $L$, and the number of factors $K$. In practice, $H$ has little influence on the estimated directions, which was first pointed out in \cite{li-1991} and can be seen from
(\ref{eq2.5c}). The choice of $H$ differs from that of a smoothing parameter in nonparametric regression, which may lead to a slower rate of convergence. As shown in Section 3, we always have the same rate of convergence for $\hphi_i$ no matter how $H$ is specified in the sufficient forecasting. The reason is that (\ref{eq2.5c}) gives the same eigenspace as long as $H\ge \max\{L,2\}$.

In terms of the choice of $L$, the first $L$ eigenvalues of $\bSigma_{f|y}$ must be significantly different from zero compared to the estimation error. Several methods such as \cite{li-1991} and \cite{Schott-1994} have been proposed to determine $L$. For instance, the average of the smallest $K-L$ eigenvalues of $\bSigma_{f|y}$ would follow a
$\chi^2$-distribution if the underlying factors are normally distributed, where $K$ is the number of underlying factors.  Should that average be large, there would be at least $L+1$ predictive indices in the sufficient forecasting model. 

Pertaining to many factor analysis, the number of latent factors $K$ might be unknown to us. There are many existing approaches to determining $K$ in the literature, e.g.,  \cite{bai-2002}, \cite{hallin-2007}, \cite{alessi-2010}. Recently, \cite{lam-yao-2012} and \cite{ahn-horenstein-2013} proposed a ratio-based estimator by maximizing the ratio of two adjacent eigenvalues of $\bX'\bX$ arranged in descending order, i.e.
$$\hat{K} = \argmax_{1\leq i \leq kmax} \hat{\lambda}_i/\hat{\lambda}_{i+1},$$
where $\hat{\lambda}_1\geq\cdots\geq\hat{\lambda}_T$ are the eigenvalues. The estimator enjoys good finite-sample performances and was motivated by the following observation: the $K$ largest eigenvalues of  $\bX'\bX$ grow unboundedly as $p$ increases, while the others remain bounded. \cite{fan-2015} proposed a similar ratio-based estimator for estimating the number of factors in semiparametric factor models. We note here that once a consistent estimator of $K$ is found, the asymptotic results in this paper hold true for the unknown $K$ case by a conditioning argument. Unless otherwise specified, we shall assume a known $K$ in the sequel.

\section{Asymptotic properties}\label{sec2.3}

We define some necessary notation here. For a vector $\bv$, $\|\bv\|$ denotes its Euclidean norm. For a matrix $\bM$, $\|\bM\|$ and $\|\bM\|_1$ represent its spectral norm and $\ell_1$ norm respectively, defined as the largest singular value of $\bM$ and its maximum absolute column sum. $\|\bM\|_\infty$ is its matrix $\ell_\infty$ norm defined as the maximum absolute row sum. For a symmetric matrix $\bM$, $\|\bM\|_1=\|\bM\|_\infty$. Denote by $\lambda_{\min}(\cdot)$ and $\lambda_{\max}(\cdot)$ the smallest and largest eigenvalue of a symmetric matrix respectively.

\subsection{Assumptions}
We first detail the assumptions on the forecasting model (\ref{eq2.1}) and the associated factor model (\ref{eq2.2}), in which common factors $\{\bff_t\}_{t=1,\ldots,T}$ and the forecasting function $h(\cdot)$ are unobserved and only $\{(\bx_t,y_t)\}_{t=1,\ldots,T}$ are observable.

\begin{assum}[Factors and Loadings] \label{a1}
For some $M>0$, \\
 (1) The loadings $\bb_i$ satisfy that $\|\bb_i\|\leq M$ for $i=1,\ldots,p$. As $p\rightarrow\infty$, there exist two positive constants $c_1$ and $c_2$ such that $$c_1<\lambda_{\min}(\frac{1}{p}\bB'\bB)<\lambda_{\max}(\frac{1}{p}\bB'\bB)<c_2.$$
(2) Identification: $T^{-1}\bF'\bF =\bI_K$, and $\bB'\bB$ is a diagonal matrix with distinct entries. \\
(3) Linearity: $E(\bb'\bff_t|\bphi_1'\bff_t,\cdots,\bphi_L'\bff_t)$ is a linear function of $\bphi_1'\bff_t,\cdots,\bphi_L'\bff_t$ for  any $\bb\in\R^p$, where $\bphi_i$'s come from model (\ref{eq2.1}).
\end{assum}

Condition (1) is often known as the \textit{pervasive} condition \citep{bai-2002,fan-2013} such that the factors impact a non-vanishing portion of the predictors. Condition (2) corresponds to the PC1 condition in \cite{bai-ng-2013}, which eliminates rotational inderterminacy in the individual columns of $\bF$ and $\bB$. Condition (3) ensures that the (centered) inverse regression curve $E(\bff_t|y_{t+1})$ is contained in the central subspace, and it is standard in the dimension reduction literature. The linearity condition is satisfied when the distribution of $\bff_t$ is elliptically symmetric \citep{eaton-1983}, and it is also asymptotically justified when the dimension of $\bff_t$ is large \citep{hall-li-1993}. 

We impose the strong mixing condition on the data generating process. Let $\mathcal{F}_{\infty}^0$ and $\mathcal{F}_T^{\infty}$ denote the $\sigma-$algebras generated by $\{(\bff_t,\bu_t,\epsilon_{t+1}): t\leq 0\}$ and $\{(\bff_t,\bu_t,\epsilon_{t+1}): t\geq T\}$ respectively. Define the mixing coefficient
$$\alpha(T)=\sup\limits_{A\in\mathcal{F}_{\infty}^0, B\in\mathcal{F}_T^{\infty}} |P(A)P(B)-P(AB)|.$$

\begin{assum}[Data generating process] \label{a2}
$\{\bff_t\}_{t\geq 1}$, $\{\bu_t\}_{t\geq 1}$ and $\{\epsilon_{t+1}\}_{t\geq 1}$ are three strictly stationary processes, and they are mutually independent. The factor process also satisfies that $E\|\bff_t\|^4<\infty$ and $E(\|\bff_t\|^2|y_{t+1}) < \infty$. In addition, the mixing coefficient $\alpha(T)<c\rho^T$ for all $T\in\mathbb{Z}^+$ and some $\rho\in(0,1)$.
\end{assum}

The independence between $\{\bff_t\}_{t\geq 1}$ and $\{\bu_t\}_{t\geq 1}$ (or $\{\epsilon_{t+1}\}_{t\geq 1}$) is similar to Assumption A(d) in \cite{bai-ng-2013}. The independence between $\{\bu_t\}_{t\geq 1}$ and $\{\epsilon_{t+1}\}_{t\geq 1}$, however, can be relaxed, making the data generation process more realistic. For example, we just need to assume that $E(\bu_t|y_{t+1}) = 0$ for $t\geq 1$ for the subsequent theories to hold. Nonetheless, we stick to this simplified assumption for clarity.


Moreover, we make the following assumption on the residuals and dependence of the factor model \eqref{eq2.2}. Such conditions are similar to those in \cite{bai-2003}, which are needed to consistently estimate the common factors as well as the factor loadings.

\begin{assum}[Residuals and Dependence] \label{a3}
There exists a positive constant $M<\infty$ that does not depend on $p$ and $T$, such that\\
(1) $E(\bu_t)=\bzero$, and $E|u_{it}|^8\leq M$. \\
(2) $\|\bSigma_u\|_1\leq M$, and for every $i,j,t,s>0$, $(pT)^{-1}\sum_{i,j,t,s}|E(u_{it}u_{js})|\leq M$  \\
(3) For every $(t,s)$, $E|p^{-1/2}( \bu_s'\bu_t-E(\bu_s\bu_t) )|^4\leq M$.
\end{assum}

\subsection{Rate of convergence}

The following theorem establishes the convergence rate of the sliced covariance estimate of inverse regression curve (i.e. $\hSig_{f|y}$ as in \eqref{eq2.9} or \eqref{eq2.10}) under the spectral norm. It further implies the same convergence rate of the estimated SDR directions associated with sufficient predictive indices. For simplicity, we assume that $K$ and $H$ are fixed, though they can be extended to allow $K$ and $H$ to depend on $T$.  This assumption enables us to obtain faster rate of convergence.  Note that the regression indices $\{\bphi_i\}_{i=1}^L$ and the PCA directions $\{\bpsi_i\}_{i=1}^L$ are typically different, but they can span the same linear space.

\begin{thm} \label{thm1}
Suppose that Assumptions \ref{a1}-\ref{a3} hold and let $\omega_{p,T}=p^{-1/2}+T^{-1/2}$.
Then under the forecasting model (\ref{eq2.1}) and the associated factor model (\ref{eq2.2}), we have
\begin{equation} \label{eq3.1}
\|\hbSig_{f|y} - \bSigma_{f|y}\|=O_p(\omega_{p,T}).
\end{equation}
If the $L$ largest eigenvalues of $\bSigma_{f|y}$ are positive and distinct, the eigenvectors $\hpsi_1,\cdots,\hpsi_L$ associated with the $L$ largest eigenvalues of $\hbSig_{f|y}$ give the consistent estimate of directions $\bpsi_1,\cdots,\bpsi_L$ respectively, with rates
\begin{equation} \label{eq3.2}
\|\hpsi_j - \bpsi_j\|=O_p(\omega_{p,T}),
\end{equation}
for $j=1,\cdots,L$, where $\bpsi_1,\cdots,\bpsi_L$ form an orthonormal basis for the central subspace $S_{y|\bff}$ spanned by $\bphi_1,\cdots,\bphi_L$.
\end{thm}

When $L=1$ and the link function $h(\cdot)$ in model \eqref{eq2.1} is linear, \cite{stock-watson-2002} concluded that having estimated factors as regressors in the linear forecasting does not affect consistency of the parameter estimates in a forecasting equation. Theorem \ref{thm1} extends the theory, showing that it is also true for any unknown link function $h(\cdot)$ and multiple predictive indices by using the sufficient forecasting. The two estimates coincide in the linear forecasting case, but our procedure can also deal with the nonlinear forecasting cases.

\begin{remark}
Combining Theorem \ref{thm1} and \cite{fan-2015}, we could obtain the rate of convergence for sufficient forecasting using semiparametric factor models (i.e., \eqref{eq2.11}-\eqref{eq2.12}). Under some conditions, when $J=(p \min\{T,p,v^{-1}_p\})^{1/\kappa}$ and $p,T\rightarrow\infty$, we have
$$
\|\hbSig_{f|y} - \bSigma_{f|y}\|=O_p(p^{-1/2}\cdot (p \min\{T,p,v^{-1}_p\})^{\frac{1}{2\kappa}-\frac{1}2}+T^{-1/2}),
$$
where $\kappa\ge 4$ is a constant characterizing the accuracy of sieve approximation (see Assumption 4.3 of \cite{fan-2015}), and $v_p=\max_{k}\frac{1}p\sum_{i}\var(\gamma_{ik})$ . Then, for $j=1,\cdots,L$, we have
$$
\|\hpsi_j - \bpsi_j\|=O_p(p^{-1/2}\cdot (p \min\{T,p,v^{-1}_p\})^{\frac{1}{2\kappa}-\frac{1}2}+T^{-1/2}),
$$
where $\bpsi_1,\cdots,\bpsi_L$ form an orthonormal basis for $S_{y|\bff}$, when the $L$ largest eigenvalues of $\bSigma_{f|y}$ are positive and distinct. It is obvious that $(p \min\{T,p,v^{-1}_p\})^{\frac{1}{2\kappa}-\frac{1}2}$ $\rightarrow 0$ even when $T$ is finite. Thus, with covariates $\bZ$, the sufficient forecasting achieves a faster rate of convergence by using the projected PCA \citep{fan-2015}. For space consideration, we do not present technical details in this paper.
\end{remark}

As a consequence of Theorem \ref{thm1}, the sufficient predictive indices can be consistently estimated as well. We present this result in the following corollary.

\begin{cor} \label{cor3.1}
Under the same conditions of Theorem \ref{thm1}, for any $j=1, \ldots, L$, we have
\begin{equation}\label{eq3.3}
\hpsi_j'\hf_t\rightarrow_p\bpsi_j\bff_t.
\end{equation}
Translating into the original predictors $\bx_t$ and letting $\hxi_j=\hLam_b'\hpsi_j$, we then have
\begin{equation}\label{eq3.4}
\hxi_j'\bx_t\rightarrow_p\bpsi_j\bff_t.
\end{equation}
\end{cor}

Corollary \ref{cor3.1} further implies the consistency of our proposed sufficient forecasting by using the classical nonparametric estimation techniques \citep{green-silverman-1993,fan-1996,li-racine-2007}.

\begin{remark} In view of \eqref{eq3.3} and \eqref{eq3.4} in Corollary \ref{cor3.1}, the sufficient forecasting finds not only the sufficient dimension reduction directions $\hpsi_1,\cdots,\hpsi_L$ for estimated factors $\hf_t$ but also the sufficient dimension reduction directions $\hxi_1,\cdots,\hxi_L$ for predictor $\bx_t$. The sufficient forecasting effectively obtains the projection of the $p$-dimensional predictor $\bx_t$ onto the $L$-dimensional subspace $\hxi_1'\bx_t,\cdots,\hxi_L'\bx_t$ by using the factor model, where the number of predictors $p$ can be larger than the number of observations $T$. It is well-known that the traditional sliced inverse regression and many other dimension reduction methods are limited to either a fixed dimension or a diverging dimension that is smaller than the sample size \citep{zhu-etal-2006}. By condensing the cross-sectional information into indices, our work alleviates what plagues the sufficient dimension reduction in high-dimensional regimes, and the factor structure in (\ref{eq2.1}) and (\ref{eq2.2}) effectively reduces the dimension of predictors and extends the methodology and applicability of the sufficient dimension reduction.
\end{remark}

\subsection{Linear forecasting under link violation}

Having extracted common factors from the vast predictors, the PCR is a simple and natural method to run a multiple regression and forecast the outcome.  Such a
linear estimate is easy to construct and provides a benchmark for our analysis. 
When the underlying relationship between the response and the latent factors is nonlinear, directly applying linear forecasts would violate the link function $h(\cdot)$ in \eqref{eq2.2}. 
But what does this  method do in the case of model misspecification? In what follows, we shall see that asymptotically the linear forecast falls into the central subspace $S_{y|\bff}$, namely, it is a weighted average of the sufficient predictive indices.

With the underlying factors $\{\hf_t\}$ estimated as in (\ref{eq2.7}) and (\ref{eq2.8}),  the response $y_{t+1}$ is then regressed on $\hf_t$ via a multiple regression.  In this case, the design matrix is orthogonal due to the normalization (\ref{eq2.8}).  Therefore, the least-squares estimate of regression coefficients from regressing $y_{t+1}$ on $\hf_t$ is
\begin{equation}\label{eq3.5}
\hphi = \frac{1}{T-1}\sum_{t=1}^{T-1}y_{t+1}\hf_t.
\end{equation}

To study the behavior of regression coefficient $\hphi$, we shall assume normality of the underlying factors $\bff_t$. The following theorem shows that, regardless of the specification of the link function $h(\cdot)$, $\hphi$ falls into the central subspace $S_{y|\bff}$ spanned by $\bphi_1,\cdots,\bphi_L$ as $p,T\rightarrow\infty$. More specifically, it converges to the sufficient direction
$$
\bar{\bphi} = \sum_{i=1}^L E((\bphi_i'\bff_t)y_{t+1})\bphi_i,
$$
which is a weighted average of sufficient directions $\bphi_1,\cdots,\bphi_L$.
In particular, when the link function is linear and $L=1$, the PCR yields an asymptotically consistent direction.

\begin{thm} \label{thm2}
Consider model (\ref{eq2.1}) and (\ref{eq2.2}) under assumptions of Theorem \ref{thm1}. Suppose the factors $\bff_t$ are normally distributed and that $E(y_t^2) < \infty$. Then, we have
\begin{equation}
\|\hphi - \bar{\bphi}\| = O_p(\omega_{p,T}).
\end{equation}
\end{thm}

As a consequence of Theorem \ref{thm2}, the linear forecast $\hphi'\hf_t$ is a weighted average of the sufficient predictive indices $\bphi_1'\bff_t,\cdots,\bphi_L'\bff_t$.
\begin{cor} \label{cor3.2}
Under the same conditions of Theorem \ref{thm2}, we have
\begin{equation}
\hphi'\hf_t \rightarrow_p \bar{\bphi}'\bff_t = \sum_{i=1}^L E((\bphi_i'\bff_t)y_{t+1})\bphi_i'\bff_t.
\end{equation}
\end{cor}

It is interesting to see that when $L=1$, $\bar{\bphi} = E((\bphi_1'\bff_t)y_{t+1})\bphi_1$. As a result, the least squares coefficient $\hphi$ in \eqref{eq3.5}, using the estimated factors as covariates,  delivers an asymptotically consistent estimate of the forecasting direction $\bphi_1$ up to a multiplicative scalar.  In this case, we can first get a scatter plot for the data $\{(\hphi'\hf_t, y_{t+1})\}_{t=1}^{T-1}$ to check if the linear fit suffices and then employ a nonparametric smoothing to get a better fit if necessary.  This can correct the bias of PCR, which delivers exactly one index of extracted factors for forecasting.

When there exist multiple sufficient predictive indices (i.e., $L\geq 2$), the above robustness no longer holds. The estimated coefficient $\hphi$ belongs to the central subspace, and it only reveals one dimension of the central subspace $S_{y|\bff}$, leading to limited predictive power. The central subspace $S_{y|\bff}$, in contrast, is entirely contained in $\bSigma_{f|y}$ as in \eqref{eq2.5}. By estimating $\bSigma_{f|y}$, the sufficient forecasting tries to recover all the effective directions and would therefore capture more driving forces to forecast.

As will be demonstrated in Section 4, the proposed method is comparable to the linear forecasting when $L=1$, but significantly outperforms the latter when $L\geq 2$ in both simulation studies and empirical applications.

\section{Numerical studies}

In this section, we conduct both Monte Carlo experiments and an empirical study to numerically assess the proposed sufficient forecasting. Sections 4.1-4.3 include three different simulation studies, and Section 4.4 shows the empirical study.

\subsection{Linear forecasting}

We first consider the case when the target is a linear function of the latent factors plus some noise. Prominent examples include asset return predictability, where we could use the cross section of book-to-market ratios to forecast aggregate market returns  \citep{campbell-1988,polk-2006,kelly-pruitt-2013}. To this end, we specify our data generating process as
\bean
& &y_{t+1} = \bphi'\bff_t + \sigma_y\epsilon_{t+1}, \\
& &x_{it}=\bb_i'\bff_t+u_{it},
\eean
where we let $K=5$ and $\bphi=(0.8,0.5,0.3,0,0)'$, namely there are 5 common factors driving the predictors and a linear combination of them predicting the outcome.  Factor loadings $\bb_i$'s are drawn from the standard normal distribution. To account for serial correlation, we set $f_{jt}$ and $u_{it}$ as two AR(1) processes respectively,
$$f_{jt}=\alpha_jf_{jt-1}+e_{jt},\quad u_{it}=\rho_i u_{it-1}+\nu_{it}.$$
We draw $\alpha_j, \rho_i$ from $\sim U[0.2,0.8]$ and fix them during simulations, while the noises $e_{jt}, \nu_{it}$ and $\epsilon_{t+1}$ are standard normal respectively. $\sigma_y$ is taken as the variance of the factors, so that the infeasible best forecast using $\bphi'\bff_t$ has an $R^2$ of approximately $50\%$.

\begin{center}
\small
\begin{threeparttable}
 \caption{Performance of forecasts using in-sample and out-sample $R^2$}
\setlength{\tabcolsep}{12pt}
 \begin{tabular}{cc |ccc|ccc}
\hline
\hline
\multicolumn{2}{c}{} &  \multicolumn{3}{c}{In-sample $R^2$ (\%)} & \multicolumn{3}{c}{Out-of-sample $R^2$ (\%)} \\
$p$ & $T$ & SF(1) & PCR & PC1 & SF(1) & PCR & PC1\\
\hline

50 & 100 & 46.9 & 47.7 &7.8 & 35.1 & 39.5 & 2.4\\
50 & 200 & 46.3 & 46.5 & 6.6 & 42.3 & 41.7 & 4.4\\
100 & 100 &49.3 &50.1 & 8.9 & 37.6 & 40.3 & 3.0\\
100 & 500 & 47.8 & 47.8 &5.5 & 43.6 & 43.5 & 1.1\\
500 & 100 & 48.5 & 48.8 & 7.9 & 40.0 & 43.1 & 4.7\\
500 & 500 & 48.2 & 48.3 & 7.2 & 48.0 & 47.9 & 6.0\\
\hline
\hline

\end{tabular}
\small
\item \textit{Notes:} In-sample and out-of-sample median $R^2$, recorded in percentage, over 1000 simulations. SF(1) denotes the sufficient forecast using one single predictive index, PCR denotes principal component regression, and PC1 uses only the first principal component.

\label{table1}
\end{threeparttable}

\end{center}

We examine both the in-sample and out-of-sample performance between the principal component regression and the proposed sufficient forecasting. Akin to the in-sample $R^2$ typically used in linear regression, the out-of-sample $R^2$, suggested by \cite{campbell-2008}, is defined here as
\begin{equation}\label{Ros}
R_{\text{OS}}^2 = 1 - \frac{\sum_{t=[T/2]}^T(y_t - \hat y_t)^2}{\sum_{t=[T/2]}^T(y_t - \bar y_t)^2},
\end{equation}
where in each replication we use the second half of the time series as the testing sample. $\hat y_t$ is the fitted value from the predictive regression using all information before $t$ and $\bar y_t$ is the mean from the test sample. Note that $R_{\text{OS}}^2$ can be negative.

Numerical results are summarized in Table \ref{table1}, where SF(1) denotes the sufficient forecast using only one single predictive index, PCR denotes the principal component regression, and PC1 uses only the first principal component.
Note that the first principal component is not necessarily the same as the sufficient regression direction $\bphi$.  Hence, PC1 can perform poorly.
As shown in Table \ref{table1}, the sufficient forecasting SF(1), by finding one single projection of the latent factors, yields comparable results as PCR, which has a total of five regressors.  The good performance of SF(1) is consistent with Theorem~\ref{thm1}, whereas the good performance of PCR is due to the fact that the true predictive model is linear (see \cite{stock-watson-2002} and Theorem~\ref{thm2}). In contrast, using the first principal component alone has very poor performance in general, as it is not the index for the regression target.

\subsection{Forecasting with factor interaction}
We next turn to the case when the interaction between factors is present. Interaction models have been employed by many researchers in both economics and statistics. For example, in an influential work, \cite{rajan-1998} examined the interaction between financial dependence and economic growth. Consider the model
$$ y_{t+1} = f_{1t}(f_{2t}+f_{3t}+1) + \epsilon_{t+1}, $$
where $\epsilon_{t+1}$ is taken to be standard normal. The data generating process for the predictors $x_{it}$'s is set to be the same as that in the previous simulation, but we let K=7, i.e., 7 factors drive the predictors with only the first 3 factors associated with the response (in a nonlinear fashion). The true sufficient directions are the vectors in the plane $S_{f|y}$ generated by $\bphi_1=(1,0,0,0,0,0,0)'$ and $\bphi_2=(0,1,1,0,0,0,0)'/\sqrt{2}$. Following the procedure in Algorithm \ref{Algo2}, we can find the estimated factors $\hf_t$ as well as the forecasting directions $\hphi_1$ and $\hphi_2$.

The forecasting model only involves a strict subset of the underlying factors. Had we made a linear forecast, all the estimated factors would likely be selected, resulting in issues such as use of irrelevant factors (i.e., only the first three factors suffice, but the first 7 factors are used). This is evident through the correlations between the response and the estimated factors, as in Table \ref{tab_corr} and results in both bias (fitting a wrong model)
and variance issue.

\begin{center}
\small
\begin{threeparttable}%
 \caption{$|\text{Corr}(y_{t+1},\hf_{it})|$}\label{tab_corr}
 \begin{tabular}{cc|ccccccc}
\hline
\hline
$p$ & $T$ & $\hf_{1t}$ &$\hf_{2t}$  & $\hf_{3t}$ & $\hf_{4t}$& $\hf_{5t}$ & $\hf_{6t}$ &$\hf_{7t}$\\
\hline
500 & 500 & 0.1819 &   0.1877  &  0.1832 &   0.1700  &  0.1737  &  0.1712   & 0.1663 \\
\hline
\hline

\end{tabular}
\small
\item \textit{Notes:} Median correlation between the forecast target and the estimated factors in absolute values, based on 1000 replications.
\end{threeparttable}

\end{center}

Next, we compare the numerical performance between the principal component regression and the proposed sufficient forecasting. As in Section 4.1, we again examine their in-sample and out-of-sample forecasting performances using the in-sample $R^2$ and the out-of-sample $R^2$ respectively. Furthermore, we will measure the distance between any estimated regression direction $\hphi$ and the central subspace $S_{f|y}$ spanned by $\bphi_1$ and $\bphi_2$. We ensure that the true factors and loadings meet the identifiability conditions by calculating an invertible matrix $\bH$ such that $\frac{1}{T}\bH\bF'\bF\bH' = \bI_K$ and $\bH^{-1}\bB'\bB\bH'^{-1}$ is diagonal. The space we are looking for is then spanned by $\bH^{-1}\bphi_1$ and $\bH^{-1}\bphi_2$. For notation simplicity, we still represent the rotated forecasting space as $S_{f|y}$. 
 We employ the squared multiple correlation coefficient $R^2(\hphi)$ to evaluate the effectiveness of $\hphi$, where
$$
R^2(\hphi) = \max_{\bphi\in S_{f|y},\|\bphi\|=1} \ {(\hphi'\bphi)^2}.
$$
This measure is commonly used in the dimension reduction literature; see \cite{li-1991}. We calculate $R^2(\hphi)$ for the sufficient forecasting directions $\hphi_{1},\hphi_2$ and the PCR direction $\hphi_{pcr}$.

\begin{center}
\small
\begin{threeparttable}%
 \caption{Performance of estimated index $\hphi$ using $R^2(\hphi)$ (\%)}\label{tab2b}
\setlength{\tabcolsep}{12pt}
 \begin{tabular}{cc|cc |c}
\hline
\hline
\multicolumn{2}{c}{}&  \multicolumn{2}{c}{Sufficient forecasting} &  \multicolumn{1}{c}{PCR} \\
$p$ & $T$ & $R^2(\hphi_1)$ & $R^2(\hphi_2)$ & \hspace*{0.3 in} $R^2(\hphi_{pcr})$ \hspace*{0.3 in}\\
\hline

100 & 100 & 84.5 (16.3) & 64.4 (25.0) & 91.4 (9.1)\\
100 & 200 & 92.9 (7.8) & 83.0 (17.5) &  95.7 (4.6)\\
100 & 500 & 96.9 (2.2) & 94.1 (5.0) & 96.7 (1.6)\\
500 & 100 & 85.4 (15.9) & 57.9 (24.8) & 91.9 (8.8)\\
500 & 200 & 92.9 (6.7) & 82.9 (16.1) & 95.9 (4.1)\\
500 & 500 & 97.0 (2.0) & 94.5 (4.5) & 98.2 (1.5)\\
\hline
\hline

\end{tabular}
\small
\item \textit{Notes:} Median squared multiple correlation coefficients in percentage over 1000 replications. The values in parentheses are the associated standard deviations. PCR denotes principal component regression.
\end{threeparttable}

\end{center}

\begin{center}
\small
\begin{threeparttable}%
 \caption{Performance of forecasts using in-sample and out-sample $R^2$}\label{tab2}
\setlength{\tabcolsep}{12pt}
 \begin{tabular}{cc |ccc|ccc}
\hline
\hline
\multicolumn{2}{c}{} &  \multicolumn{3}{c}{In-sample $R^2$ (\%)} & \multicolumn{3}{c}{Out-of-sample $R^2$ (\%)} \\
$p$ & $T$  & SFi & PCR & PCRi & SFi & PCR & PCRi\\
\hline

100 & 100  &46.2 & 38.5 & 42.4 & 20.8 & 12.7 & 13.5\\
100 & 200  & 57.7 & 35.1 &  38.6 & 41.6 & 24.0 & 24.7\\
100 & 500  & 77.0 & 31.9 & 34.9 & 69.7 & 29.1 & 31.5\\
500 & 100  & 49.5 & 38.7 & 22.7 & 21.9 & 16.6 & 19.2\\
500 & 200  & 58.9 & 34.7 & 39.0 & 40.2 & 22.2 & 24.0\\
500 & 500  & 79.8 & 32.5 & 35.6 & 72.3 & 26.9 & 28.2\\
\hline
\hline

\end{tabular}
\small
\item \textit{Notes:} In-sample and out-of-sample median $R^2$ in percentage over 1000 replications. SFi uses first two predictive indices and includes their interaction effect; PCR uses all principal components; PCRi extends PCR by including an extra interaction term built on the first two principal components of the covariance matrix of predictors.
\end{threeparttable}

\end{center}

In the sequel, we draw two conclusions from the simulation results summarized in Tables  \ref{tab2b} and \ref{tab2} based on 1000 independent replications.

Firstly, we observe from Table \ref{tab2b} that the length of time-series  $T$ has an important effect on $R^2(\hphi_1)$, $R^2(\hphi_2)$ and $R^2(\hphi_{pcr})$. As $T$ increases from 100 to 500, all the squared multiple correlation coefficients increase too.
In general, the PCR and the proposed sufficient forecasting perform very well in terms of the measure $R^2(\hphi)$. The good performance of $R^2(\hphi_{pcr})$ is a testimony of the correctness of Theorem~\ref{thm2}. However, the PCR is limited to the linear forecasting function, and it only picks up one effective sufficient direction. The sufficient forecasting successfully picks up both effective sufficient directions in the forecasting equation, which is consistent to Theorem~\ref{thm1}.

Secondly, we evaluate the power of sufficient forecasting using the two predictive indices $\hphi_1'\hf_t$ and $\hphi_2'\hf_t$. Similar to the previous section, we build a simple linear regression to forecast, but include three terms, $\hphi_1'\hf_1$, $\hphi_2'\hf_2$ and $(\hphi_1'\hf_1)\cdot( \hphi_2'\hf_2)$,  as regressors to account for interaction effects. In-sample and out-of-sample $R^2$ are reported. For comparison purposes, we also report results from principal component regression (PCR), which uses the 7 estimated factors extracted from predictors. In addition, we add the interaction between the first two principal components on top of the PCR, and denote this method as PCRi. Note that the principal component directions differ from the sufficient forecast directions, and it does not take into account the target information.  Therefore, PCRi model is very different from the sufficient regression with interaction terms.  As shown in Table \ref{tab2}, the in-sample $R^2$'s of PCR hover around 35\%, and its out-of-sample $R^2$'s are relatively low.  Including interaction between the first two PCs of the covariance of the predictors does not help much, as the interaction term is incorrectly specified for this model. SFi picks up the correct form of interaction and exhibit better performance, especially when $T$ gets reasonably large.

\subsection{Forecasting using semiparametric factor models}
When the factor loadings admit a semiparametric structure, this additional information often leads to improved estimation accuracy. In this section, we consider a simple design with one observed covariate and three latent factors. We set the three loading functions as $g_1(z)=z$, $g_2(z)=z^2-1$ and $g_3(z) = z^3 -2z$, where the characteristic $z$ is drawn from standard normal. We examine the same predictive model as outlined in Section 4.2, which involves factor interaction. The only difference now is that our loadings are covariate-dependent and we no longer have irrelevant factors.

\begin{figure}[!htbp]
\begin{center}
\includegraphics[scale=0.63]{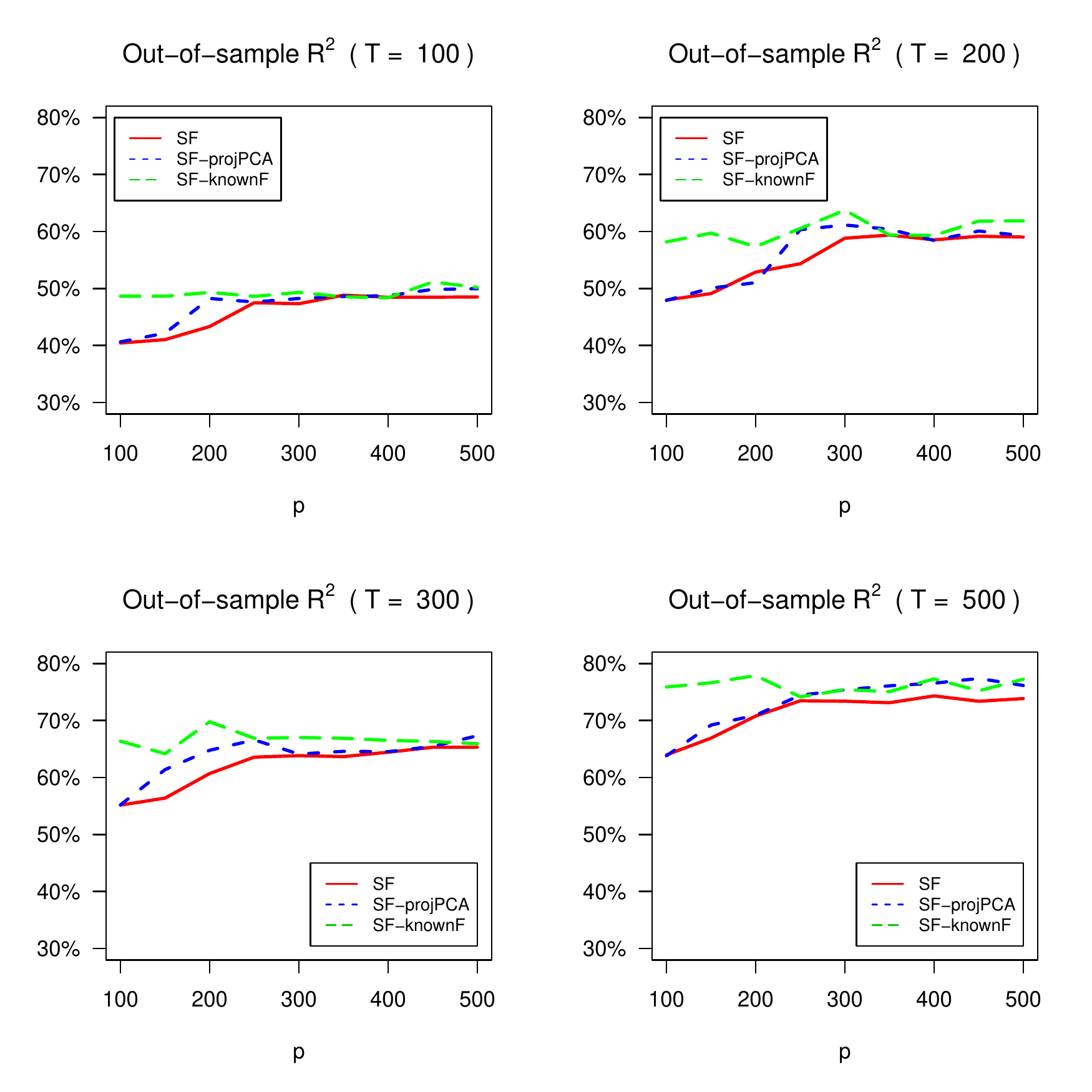}
\caption[Forecasting results for CCINRV (consumer credit outstanding)]{\small Out-of-sample $R^2$ over 1000 replications. SF denotes the sufficient forecasting method without using extra covariate information. SF-projPCA uses projected-PCA method to estimate the latent factors, and SF-knownF uses the true factors from the data generating process. All the methods are based on 2 predictive indices. }
\end{center}
\label{fig2}
\end{figure}

Figure 2 presents the out-of-sample performance of three methods for different combinations of $(p,T)$. The three methods differ in their ways of estimating the latent factors. In the first method, we simply use traditional PCA to extract underlying factors, whereas in the second method we resort to projected-PCA. For comparison, we also look at the performance when the true factors are used to forecast the target. In general, as $T$ increases, the predictive power significantly increases, irrespective of the cross-sectional dimension $p$. For fixed $T$, the projected-PCA quickly picks up predictive power as $p$ increases. Using extra covariate information helps up gain accuracy in estimating the underlying factors, and lifts up the out-of-sample $R^2$ very close to those obtained from using the true factors. This further demonstrates that with the presence of known characteristics in factor loadings, the projected-PCA would buy us more power in forecasting the factor-driven target.

\subsection{An empirical example}

As an empirical investigation, we use sufficient forecasting to forecast several macroeconomic variables. Our dataset is taken from \cite{stock-watson-2012}, which consists of quarterly observations on 108 U.S. low-level disaggregated macroeconomic time series from 1959:I to 2008:IV. Similar datasets have been extensively studied in the forecasting literature \citep{bai-2008, Ludvigson-Ng-2009}. These time series are transformed by taking logarithms and/or differencing to make them approximately stationary. We treat each of them as a forecast target $y_{t}$, with all the others forming the predictor set $\bx_t$. The out-of-sample performance of each time series is examined, measured by the out-of-sample $R^2$ in \eqref{Ros}. The procedure involves fully recursive factor estimation and parameter estimation starting half-way of the sample, using data from the beginning of the sample through quarter $t$ for forecasting in quarter $ t + 1$.

We present the forecasting results for a few representatives in each macroeconomic category. The time series are chosen such that the second eigenvalue of $\hbSig_{f|y}$ exceeds $50\%$ of the first eigenvalue in the first half training sample. This implies that a single predictive index may not be enough to forecast the given target, so we could consider the effect of second predictive index. Note that in some categories, the second eigenvalues $\hbSig_{f|y}$ for all the time series are very small; in this case, we randomly pick a representative.  As described in previous simulations, we use SF(1) to denote sufficient forecasting with one single predictive index, which corresponds to a linear forecasting equation with one regressor. SF(2) employs an additional predictive index on top of SF(1), and is fit by local linear regression. In all cases, we use 7 estimated factors extracted from 107 predictors, which are found to explain over 40 percent of the variation in the data; see \cite{bai-ng-2013}.

\begin{center}
\small
\begin{threeparttable}
 \caption{Out-of-sample Macroeconomic Forecasting}
\setlength{\tabcolsep}{11pt}
\begin{tabular}{llcc cc}
\hline
\hline
Category & Label & SF(1) & SF(2) & PCR & PC1 \\
\hline

GDP components & GDP264 & 11.2 & 14.9 & 13.8 & 9.2 \\
IP & IPS13 & 17.8 & 21.0 & 20.6& -2.9 \\
Employment & CES048  & 24.1 & 26.4 & 23.4 & 25.7 \\
Unemployment rate & LHU680 & 19.0 & 21.7 & 13.6 & 21.1 \\
Housing & HSSOU & 1.9 & 4.4 &-0.8 & 7.4\\
Inventories & PMNO & 19.8 & 20.2 & 18.6 &6.3 \\
Prices & GDP275\_3 & 9.1 & 11.3 & 10.2 & 0.2 \\
Wages & LBMNU & 32.7 & 34.2 &34.0 & 29.1 \\
Interest rates & FYFF & 3.6 & 2.4 & -31.5 & 1.6 \\
Money & CCINRV &  4.8 & 16.0 & 1.3 & -0.5 \\
Exchange rates & EXRCAN & -4.7 & 4.5 & -7.9 & -1.3 \\
Stock prices & FSDJ & -9.2 & 8.8 & -14.5& -7.7 \\
Consumer expectations & HHSNTN &-3.7 & -5.0 & -4.4 & 0.0 \\
\hline
\hline

\end{tabular}
\small
\item \textit{Notes:} Out-of-sample $R^2$ for one--quarter ahead forecasts. SF(1) uses a single predictive index built on 7 estimated factors to make a linear forecast. SF(2) is fit by local linear regression using the first two predictive indices. PCR uses all 7 principal components as linear regressors. PC1 uses the first principal component.
\label{table3}
\end{threeparttable}

\end{center}


Table \ref{table3} shows that SF(1) yields comparable performance as PCR. There are cases where SF(1) exhibits more predictability than PCR, e.g., Inventories and Interest rates. This is due to the fact that there are possible nonlinear effects from extracted factors on the target.  In this case, the first predictive index obtained from our procedure accounts for such a nonlinear effect whereas the index created by PCR finds only the best linear approximation and can not find a second index.
SF(2) improves predictability in many cases, where the target time series might not be linear in the latent factors. Taking CCINRV (consumer credit outstanding) for example, Figure 3 plots the eigenvalues of its corresponding $\hbSig_{f|y}$, the estimated regression surface and the running out-of-sample $R^2$'s by different sample split date. As can been seen from the plot, there is a non-linear effect of the two underlying macro factors on the target. By taking such effect into account, SF(2) consistently outperforms the other methods.

\begin{figure}[!htbp]
\begin{center}
\includegraphics[scale=0.8]{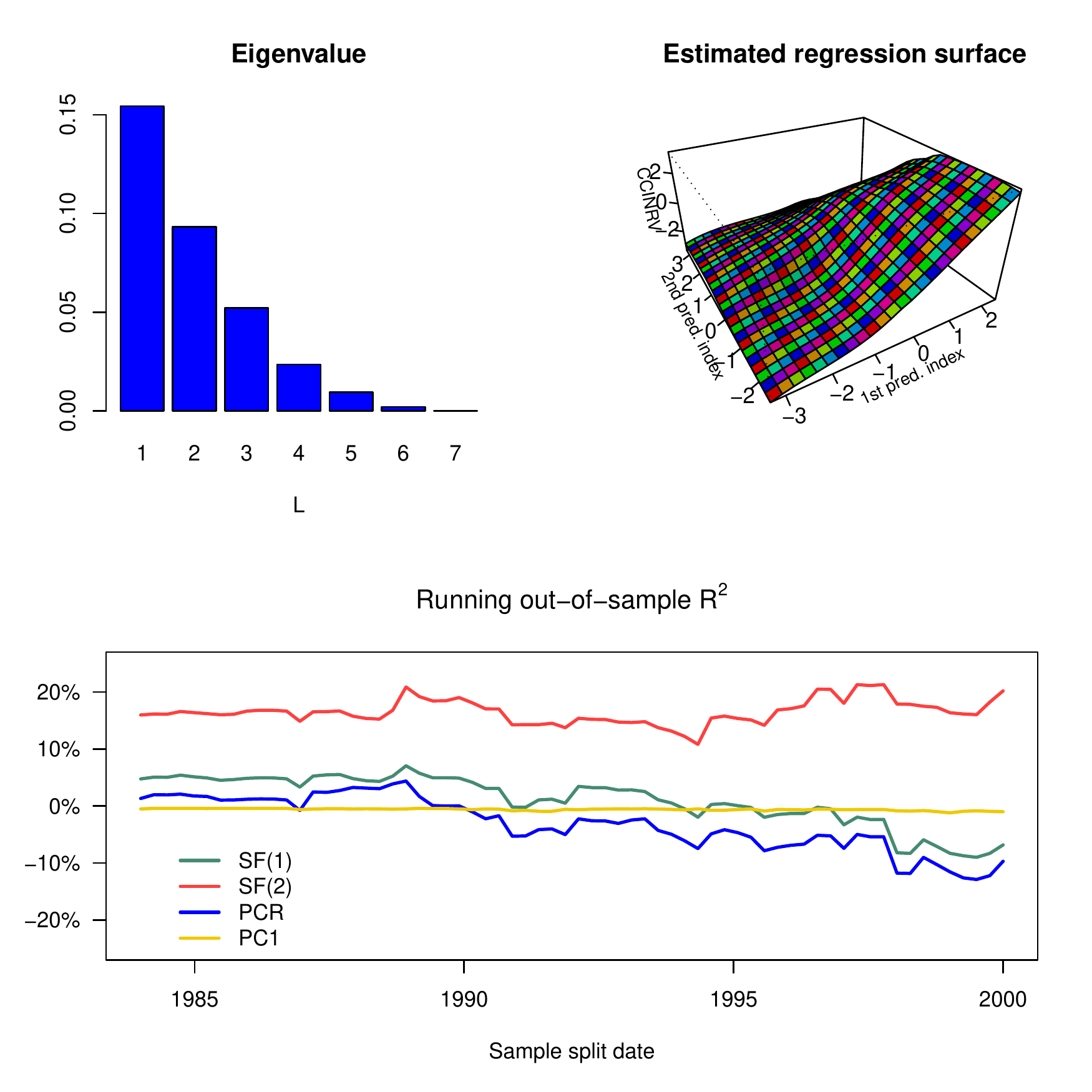}
\caption[Forecasting results for CCINRV (consumer credit outstanding)]{\small Forecasting results for CCINRV (consumer credit outstanding). The top left panel shows the eigenvalues of $\hbSig_{f|y}$. The top right panel gives a 3-d plot of the estimated regression surface. The lower panel displays the running out-of-sample $R^2$'s for four methods described in Table \ref{table3}. }
\end{center}
\label{fig3}
\end{figure}

\section{Discussion}

We have introduced the sufficient forecasting in a high-dimensional environment to forecast a single time series, which enlarges the scope of traditional factor forecasting. The key feature of the sufficient forecasting is its ability in extracting multiple predictive indices and providing additional predictive power beyond the simple linear forecasting, especially when the target is an unknown nonlinear function of underlying factors. {We explicitly point out the interesting connection between the proposed sufficient forecasting and the deep learning \citep{bengio-2009,bengio-etal-2013}.}
Furthermore, the proposed method extends the sufficient dimension reduction to high-dimensional regimes by condensing cross-sectional information through (semiparametric) factor models. We have demonstrated its efficacy through Monte Carlo experiments. Our empirical results on macroeconomic forecasting suggest that such procedure can contribute to substantial improvement beyond conventional linear models.

\section*{Acknowledgement}

The authors are grateful to the co-editor and referees for their constructive comments. The authors also thank Stephane Bonhomme, A. Ronald Gallant, Yuan Liao, Rosa Matzkin, Jose Luis Montiel Olea, and participants at Penn State University (Econometrics), New York University (Stern), Peking University (Statistics), University of South Carolina (Statistics), and The Institute for Fiscal Studies for their helpful suggestions. This paper is supported by National Institutes of Health grants R01-GM072611 and R01GM100474-04 and National Science Foundation grants DMS-1206464, DMS-1308566, and DMS-1505256.

\small
\singlespacing
\section{Appendix}

We first cite two lemmas from \cite{fan-2013}, which are needed subsequently in the proofs.
\begin{lem} \label{lem1}
Suppose $\bA$ and $\bB$ are two semi-positive definite matrices, and that $\lambda_{\min}(\bA)>c_{p,T}$ for some sequence $c_{p,T}>0$. If $\|\bA-\bB\|=o_p(c_{p,T})$, then
$$\|\bA^{-1}-\bB^{-1}\|=O_p(c_{p,T}^{-2})\|\bA-\bB\|.$$
\end{lem}

\begin{lem} \label{lem1b}
 Let $\{\lambda_i\}_{i=1}^p$ be the eigenvalues of $\bSigma$ in descending order and $\{\bxi\}_{i=1}^p$ be their associated eigenvectors. Correspondingly, let $\{\hlam_i\}_{i=1}^p$ be the eigenvalues of $\hSig$ in descending order and $\{\hxi\}_{i=1}^p$ be their associated eigenvectors. Then,
 \begin{itemize}
   \item [(a)] Weyl's Theorem: $$|\hlam_i - \lambda_i|\leq \|\hSig-\bSigma\|.$$
   \item [(b)] $\sin(\theta)$ Theorem \citep{davis-kahan-1970}:
$$\|\hxi_i-\bxi_i\|\leq\frac{\sqrt{2}\cdot \|\hSig-\bSigma\|}{\min(|\hlam_{i-1}-\lambda_i|,|\lambda_i-\hlam_{i+1}|)}.$$
 \end{itemize}
\end{lem}

\subsection{Proof of Proposition \ref{prop2.1}}

\begin{proof}
It suffices to show that $\hf_t = \hLam_b\bx_t$, or equivalently, $\hF'=\hLam_b\bX$ in matrix form. First we let $\bM=\diag(\lambda_1,\cdots,\lambda_K)$, where $\lambda_i$ are the largest $K$ eigenvalues of $\bX'\bX$. By construction, we have $(\bX'\bX)\hF = \hF \bM$.
Thus,
$$
    \hB'\hB = T^{-2}\hF'(\bX'\bX)\hF = T^{-2}\hF'\hF \bM=T^{-1}\bM.
$$
Now, it is easy to see that
$$
(\hB'\hB)^{-1}\hB'\bX  =
T \bM^{-1} (T^{-1}\hF'\bX')\bX = \bM^{-1} (\bX'\bX \hF)' = \hF'.$$
This concludes the proof of Proposition \ref{prop2.1}.

\end{proof}

\subsection{Proof of Theorem \ref{thm1}}
Let $\bV$ denote the $K\times K$ diagonal matrix consisting of the $K$ largest eigenvalues of the sample covariance matrix $T^{-1}\bX'\bX$ in the descending order. Define a $K\times K$ matrix
\begin{equation}\label{def_H}
\bH=(1/T)\bV^{-1}\hF'\bF\bB'\bB,
\end{equation}
where $\bF'=(\bff_1,\cdots,\bff_T)$. 
In factor models, the rotation matrix $\bH$ is typically used for identifiability. With identification in Assumptions \ref{a1} (2), \cite{bai-ng-2013} showed that $\bH$ is asymptotically an identity matrix. The result continues to hold under our setting.

\begin{lem} \label{lem2}
Under Assumptions \ref{a1} (1)--(2), \ref{a2} and \ref{a3}, we have
$$||\bH-\bI_K||=O_p(\omega^2_{p,T})$$
\end{lem}
\begin{proof}
In view of Assumption A and Appendix B of \cite{bai-ng-2013}, we only need to verify that Assumption A (c.ii) in \cite{bai-ng-2013} is satisfied under our models and assumptions, since all other assumptions of theirs are met here. To verity Assumption A (c.ii), it is enough to show that $\sum_{t=1}^T|E(u_{it}u_{js})|\leq M_1$ and $\sum_{i=1}^p|E(u_{it}u_{js})|\leq M_2$ for some positive constants $M_1,M_2<\infty$.

Consider $E(u_{it}u_{js})$. Since $E|u_{it}|^8\leq M$ for some constant $M$ and any $t$ by Assumption \ref{a3}(1), by using Davydov's inequality (Corollary 16.2.4 in \cite{Athreya-2006}), there exist constants $C_1$ and $C_2$ such that for all $(i,j)$,
$$|E(u_{it}u_{js})|\leq C_1\left(\alpha\left(|t-s|\right)\right)^{1-\frac{1}8-\frac{1}8}\leq C_2\rho^{\frac{3}4|t-s|},$$
where $\alpha(\cdot)$ is the mixing coefficient, $\rho\in(0,1)$, and $\alpha\left(|t-s|\right)<c\rho^{|t-s|}$ by Assumption \ref{a2}. Hence, with $M_1= 2C_2/(1-\rho^{\frac{3}4})$, we have
$$
\sum_{t=1}^T|E(u_{it}u_{js})|=\sum_{t=1}^T C_2\rho^{\frac{3}4|t-s|}\leq 2C_2/(1-\rho^{\frac{3}4})= M_1.
$$

In addition, it is easy to verify that there is a constant $C_3$ such that $$|E(u_{it}u_{js})|\leq C_3|E(u_{it}u_{jt})| = C_3|\bSigma_u(i,j)|.$$ By assumption, we also have $$\sum_{i=1}^p|E(u_{it}u_{js})|\le C_3\sum_{i=1}^p|\bSigma_u(i,j)| \leq C_3\|\bSigma_u\|_1\leq C_3M.$$ Thus, Assumption A (c.ii) in \cite{bai-ng-2013} is satisfied with $M_2=C_3M$. Therefore, with identification in Assumption \ref{a1}(2), we have
$$
    ||\bH-\bI_K|| = O_p(\omega^2_{p,T}),
$$
which completes the proof of Lemma \ref{lem2}.

\end{proof}

The next lemma shows that the normalization matrix $\bLam_b$ can be consistently estimated by the plug-in estimate $\hLam_b$ under the spectral norm, which is an important result to prove Theorem \ref{thm1}.

\begin{lem} \label{lem3}
Under Assumptions \ref{a1} (1)--(2), \ref{a2} and \ref{a3}, we have
\begin{itemize}
  \item [(a)]$\|\hB-\bB\|=O_p(p^{1/2}\omega_{p,T}),$
  \item [(b)]$\|\hLam_b-\bLam_b\|=O_p(p^{-1/2}\omega_{p,T}).$
\end{itemize}
\end{lem}

\begin{proof}
(a) First of all,
\bean
\|\hB-\bB\| &=& \|\hB-\bB\bH' + \bB(\bH' -\bI_K)\|\\
& \leq& \|\hB-\bB\bH'\| + \|\bB(\bH' -\bI_K)\|, 
\eean
where $\bH$ is defined in \eqref{def_H}.
The second term is bounded by $\|\bB\|\cdot\|\bH' -\bI_K\| = O_p(p^{1/2}\omega_{p,T}^2) = o_p(p^{1/2}\omega_{p,T})$ by Lemma \ref{lem2} and using the fact $\|\bB\|^2=\lambda_{\max}(\bB'\bB)<c_2p$ from Assumption 3.1.  It remains to bound the first term.  Using the fact that $\|\cdot\|\le\|\cdot\|_F$, we have
\begin{equation} \label{fan1}
\|\hB-\bB\bH'\|^2 \leq\|\hB-\bB\bH'\|_F^2 =\sum_{i=1}^p\|\hb_i-\bH\bb_i\|^2
\end{equation}

Next, note that both $\hb_i=(1/T)\sum_{t=1}^Tx_{it}\hf_t$ and $(1/T)\sum_{t=1}^T\hf_t\hf_t'=\bI_K$ hold as a result of the constrained
least-squares problem \eqref{eq2.7} and \eqref{eq2.8}. Then, we have
\begin{eqnarray}
\hb_i-\bH\bb_i
&=&\frac{1}{T}\sum_{t=1}^T(\hf_tx_{it}-\bH\bff_tx_{it})+\frac{1}{T}\sum_{t=1}^T\bH\bff_t(\bff_t'\bb_i+u_{it}) - \bH\bb_i \nonumber \\
&=&\frac{1}{T}\sum_{t=1}^Tx_{it}(\hf_t-\bH\bff_t)+\frac{1}{T}\sum_{t=1}^T\bH\bff_tu_{it}
\label{fan2}
\end{eqnarray}
where we have used the fact that $\frac{1}{T}\sum_{t=1}^T\bff_t\bff_t'-\bI_K=0$. The remaining two terms on the right-hand side can be separately bounded as follows.

Under Assumptions \ref{a1} (1)--(2), \ref{a2} and \ref{a3}, we have the following rate of convergence of the estimated factors,
$$
    \frac{1}{T}\sum_{t=1}^T\|\hf_t-\bH\bff_t\|^2=O_p(\omega^2_{p,T}).
$$
This result is proved in Theorem 1 of \cite{bai-2002}.

For the first term in \eqref{fan2}, it is easy to see that $E(x_{it}^2)=O(1)$ and thus $T^{-1}\sum_{t=1}^Tx_{it}^2=O_p(1)$. Hence, we use the Cauchy--Schwarz inequality to yield
\ben
\|\frac{1}{T}\sum_{t=1}^Tx_{it}(\hf_t-\bH\bff_t)\|\leq(\frac{1}{T}\sum_{t=1}^Tx_{it}^2)^{1/2}(\frac{1}{T}\sum_{t=1}^T\|\hf_t-\bH\bff_t\|^2)^{1/2}=O_p(\omega_{p,T}).
\een

For the second term, by using the Cauchy--Schwarz inequality again, we have
$$\|\frac{1}{T}\sum_{t=1}^T\bH\bff_tu_{it}\|\leq \|\bH\|\cdot\|\frac{1}{T}\sum_{t=1}^T\bff_tu_{it}\|.$$

To bound $\|\hb_i-\bH\bb_i\|^2$, we use the simple fact that $(a+b)^2\leq 2(a^2+b^2)$. Therefore we combine two upper bonds to obtain that
\bean
\|\hB-\bB\bH'\|^2
&\le& 2\sum_{i=1}^p\left(\|\frac{1}{T}\sum_{t=1}^Tx_{it}(\hf_t-\bH\bff_t)\|^2+\|\frac{1}{T}\sum_{t=1}^T\bH\bff_tu_{it}\|^2\right)\\
&\leq&2\sum_{i=1}^p\left(O_p(\omega_{p,T}^2)+O_p(1)\cdot (T^{-1}\|\frac{1}{\sqrt{T}}\sum_{t=1}^T\bff_tu_{it}\|^2)\right).
\eean
To bound $\sum_{i=1}^p\|\frac{1}{\sqrt{T}}\sum_{t=1}^T\bff_tu_{it}\|^2$, as pointed out in the proof of Lemma 4 of \cite{bai-2002}, we only need to prove that $E(\frac{1}{p}\sum_{i=1}^p\|\frac{1}{\sqrt{T}}\sum_{t=1}^T\bff_tu_{it}\|^2)\leq M_1$ for some positive constant $M_1$. Next, we use the Cauchy--Schwarz inequality and the independence between $\{\bff_t\}_{t\geq 1}$ and $\{\bu_t\}_{t\geq 1}$ to obtain
\bean
E(\frac{1}{p}\sum_{i=1}^p\|\frac{1}{\sqrt{T}}\sum_{t=1}^T\bff_tu_{it}\|^2)
&=& \frac{1}{p}\sum_{i=1}^p\frac{1}{{T}}\sum_{s=1}^T\sum_{t=1}^TE(u_{is}u_{it})E(\bff_s'\bff_t)\\
&\le& \frac{1}{pT}\sum_{i=1}^p\sum_{s=1}^T\sum_{t=1}^T|E(u_{is}u_{it})|\cdot \sqrt{E\|\bff_s\|^2E\|\bff_t\|^2}\\
&\le& M \cdot C \ (\equiv M_1)
\eean
where $E\|\bff_s\|^2 \le C$ holds due to Assumption \ref{a2}, and $\frac{1}{pT}\sum_{i=1}^p\sum_{s=1}^T\sum_{t=1}^T|E(u_{is}u_{it})|\le M$ holds due to Assumption \ref{a3}(2). Then, we can follow \cite{bai-2002} to show that $$\sum_{i=1}^p\|\frac{1}{\sqrt{T}}\sum_{t=1}^T\bff_tu_{it}\|^2=O_p(p).$$
Now, it immediately follows from \eqref{fan1} that
$$\|\hB-\bB\bH'\|^2=O_p(p\omega_{p,T}^2)+O_p(p/T)=O_p(p\omega_{p,T}^2).$$

\medskip
(b) Begin by noting that $$\|\hB\|\le \|\hB-\bB\|+\|\bB\|=O_p(p^{\frac{1}2}\omega_{p,T})+O_p(p^{1/2})=O_p(p^{1/2}).$$
We use the triangle inequality to bound
\bean
\|\hB'\hB-\bB'\bB\|
&\le& \|\hB'\hB-\bB'\hB\|+\|\bB'\hB-\bB'\bB\| \\
&\leq& \|\hB-\bB\|\cdot \|\hB\|+\|\bB'\|\cdot\|\hB-\bB\|
\eean
Note that $\|\hB'\hB-\bB'\bB\|=O_p(p\omega_{p,T})=o_p(p)$ and $\lambda_{\min}(\bB'\bB)>c_1p$ hold by Assumption 3.1. Now, we use Lemma \ref{lem1} to obtain
\ben
\|(\hB'\hB)^{-1}-(\bB'\bB)^{-1}\|= O_p(p^{-2})\cdot \|\hB'\hB-\bB'\bB\|=O_p(p^{-1}\omega_{p,T}).
\een

It follows that
\bean
\|\hLam_b-\bLam_b\|&=&\|(\hB'\hB)^{-1}\hB'-(\bB'\bB)^{-1}\bB'\| \\
&=& \|(\hB'\hB)^{-1}\hB'-(\bB'\bB)^{-1}\hB'\|+\|(\bB'\bB)^{-1}\hB'-(\bB'\bB)^{-1}\bB'\| \\
&\leq& \|(\hB'\hB)^{-1}-(\bB'\bB)^{-1}\|\cdot\|\hB\|+\|(\bB'\bB)^{-1}\|\cdot\|\hB'-\bB'\| \\
&=& O_p(p^{-1}\omega_{p,T})\cdot O_p(p^{1/2})+ O_p(p^{-1})\cdot O_p(p^{1/2}\omega_{p,T})\\
&=&O_p(p^{-1/2}\omega_{p,T}).
\eean
which completes the proof of Lemma \ref{lem3}.

\end{proof}

The following lemma lays the foundation of inverse regression, which can be found in \cite{li-1991}.
\begin{lem} \label{lem4}
Under model (\ref{eq2.1}) and Assumption \ref{a1} (3), the centered inverse regression curve $E(\bff_t|y_{t+1})-E(\bff_t)$ is contained in the linear subspace spanned by $\bphi_k'\cov(\bff_t), k=1,\ldots,L$.

\end{lem}

We are now ready to complete the proof of Theorem \ref{thm1}.

\begin{proof}[Proof of Theorem \ref{thm1}]

In the first part, we derive the probability bound for $\|\hbSig_{f|y}-\bSigma_{f|y}\|$. Let $\hm_h=\frac{1}{c}\sum_{l=1}^c\bx_{(h,l)}$ denote the average of the predictors within the corresponding slice $I_h$, and $\bm_h=E(\bx_t|y_{t+1}\in I_h)$ represents its population version.
By definition, $\bSigma_{f|y}=H^{-1}\sum_{h=1}^H(\bLam_b\bm_h)(\bLam_b\bm_h)'$. For fixed $H$, note that
$$\hbSig_{f|y}-\bSigma_{f|y}=H^{-1}\sum_{h=1}^H[(\hLam_b\hm_h)(\hLam_b\hm_h)'-(\bLam_b\bm_h)(\bLam_b\bm_h)'].$$
Thus,
\bea
\|\hbSig_{f|y}- \bSigma_{f|y} \|
&\le& H^{-1}\sum_{h=1}^H \left(\|\hLam_b\hm_h- \bLam_b\bm_h\|\cdot\|\hLam_b\hm_h\|+\|\bLam_b\bm_h\|\cdot\|\hLam_b\hm_h- \bLam_b\bm_h\|\right) \label{fan3}
\eea
We now bound each of the above terms.

From the definition, we immediately have
\bean
\|\hm_h-\bm_h\|&=&\|\frac{1}{c}\sum_{l=1}^c(\bB\bff_{(h,l)}+\bu_{(h,l)})-\bB E(\bff_t|y_{t+1}\in I_h)\|  \\
&\leq& \|\bB\|\cdot\|\frac{1}{c}\sum_{l=1}^c\bff_{(h,l)}-E(\bff_t|y_{t+1}\in I_h)\|+\|\frac{1}{c}\sum_{l=1}^c\bu_{(h,l)}\|.
\eean
Under Assumption 3.2, the sample mean $\frac{1}{c}\sum_{l=1}^c\bff_{(h,l)}$ converges to the population mean $E(\bff_t|y_{t+1}\in I_h)$ at the rate of $O_p(T^{-1/2})$.  This holds true as the random variable $\bff_t|y_{t+1}\in I_h$ is still stationary with finite second moments, and the sum of the $\alpha$-mixing coefficients converges (Theorem 19.2 in \cite{billingsley-1999}). This applies to $\bu_t|y_{t+1}\in I_h$ as well.
Therefore,
\bean
\|\hm_h-\bm_h\| &=& O_p(p^{\frac{1}2})\cdot O_p(T^{-\frac{1}2})+O_p(T^{-\frac{1}2})
= O_p(\sqrt{p/T}).
\eean

In addition to the inequality above, we also have
$$\|\bm_h\|=\|E(\bx_t|y_{t+1}\in I_h)\|\leq \|\bB\|\cdot\|E(\bff_t|y_{t+1}\in I_h)\|=O_p(p^{1/2}).$$
Thus, it follows that $$\|\hm_h\|\le \|\hm_h-\bm_h\|+\|\bm_h\| =O_p(p^{1/2}).$$
Now, we use both the triangle inequality and the Cauchy--Schwarz inequality to obtain that
\bean
\|\hLam_b\hm_h- \bLam_b\bm_h\|
&\le& \|\hLam_b\hm_h-\bLam_b\hm_h\|+\|\bLam_b\hm_h-\bLam_b\bm_h\|\\
&\leq&\|\hLam_b-\bLam_b\|\cdot\|\hm_h\|+\|\bLam_b\|\cdot\|\hm_h-\bm_h\|\\
&=&O_p(p^{-1/2}\omega_{p,T})\cdot O_p(p^{1/2})+O_p(p^{-1/2})\cdot O_p(\sqrt{p/T})\\
&=&O_p(\omega_{p,T}),
\eean
where we use $\|\hLam_b-\bH\bLam_b\|=O_p(p^{-1/2}\omega_{p,T})$ by Lemma \ref{lem3} in the second inequality, and that $\|\bLam_b\|\le  \|(\bB'\bB)^{-1}\|\cdot\|\bB\|=O_p(p^{-1/2})$ in the third equality.

Also we have that $$\|\hLam_b\hm_h\|\le(\|\hLam_b-\bLam_b\|+\|\bLam_b\|)\cdot\|\hm_h\|=O_p(p^{-1/2})\cdot O_p(p^{1/2})=O_p(1)$$ and $$\|\bLam_b\bm_h\|\le \|\bLam_b\|\cdot\|\bm_h\|=O_p(1).$$
Hence, from \eqref{fan3}, we conclude the desired result as follows:
\bean
\|\hbSig_{f|y}- \bSigma_{f|y} \|
&=& H^{-1}\sum_{h=1}^H \left(O_p(\omega_{p,T})\cdot O_p(1)+O_p(1)\cdot O_p(\omega_{p,T})\right)
= O_p(\omega_{p,T}).
\eean

In the second part, we show that $\bpsi_1,\ldots,\bpsi_L$ are the desired SDR directions and further derive the probability bound for $\|\hpsi_j-\bpsi_j\|$ with $j=1,\ldots, L$. Let $\{\lambda_i\}_{i=1}^K$ be the eigenvalues of $\bSigma_{f|y}$ in the descending order and $\{\hlam_i\}_{i=1}^K$ be the eigenvalues of $\hbSig_{f|y}$ in the descending order. Let $\hpsi_j$ be the eigenvector associated with the eigenvalue $\hlam_j$, and $\bpsi_j$ is the eigenvector associated with the eigenvalue $\lambda_j$.

In view of Lemma \ref{lem4}, $E(\bff_t|y_{t+1})$ is contained in the central subspace $S_{y|\bff}$ spanned by $\bphi_1, \ldots, \bphi_L$, since we have the normalization $\cov(\bff_t)=\bI_K$ and $E(\bff_t)=\bzero$. Recall the fact that the eigenvectors of $\bSigma_{f|y}$ corresponding to its $L$ largest eigenvalues would form another orthogonal basis for $S_{y|\bff}$. Hence, the eigenvectors $\bpsi_1, \ldots, \bpsi_L$ of $\bSigma_{f|y}$ constitute the SDR directions for model (\ref{eq2.1}).

Note that $|\hlam_{j-1}-\lambda_{j-1}|\le \|\hbSig_{f|y}-\bSigma_{f|y}\|$ holds by Weyl's Theorem in Lemma \ref{lem1b} (a). Similarly, we have $|\lambda_j-\hlam_{j+1}|\ge |\lambda_j-\lambda_{j+1}|-O_p(\omega_{p,T})$. Then, we have
\bean
|\hlam_{j-1}-\lambda_j|
&\ge& |\lambda_{j-1}-\lambda_j|-|\hlam_{j-1}-\lambda_{j-1}|\\
&\ge& |\lambda_{j-1}-\lambda_j|-O_p(\omega_{p,T}),
\eean
where the fact that $|\hlam_{j-1}-\lambda_{j-1}|\le \|\hbSig_{f|y}-\bSigma_{f|y}\|$ is used.

Since $\bSigma_{f|y}$ have distinct eigenvalues, both $|\hlam_{j-1}-\lambda_j|$ and  $|\lambda_j-\hlam_{j+1}|$ are bounded away from zero with probability tending to one for $j=1,\ldots, L$. Now, a direct application of $\sin(\theta)$ Theorem \citep{davis-kahan-1970} in Lemma \ref{lem1b} (b) shows that
\ben
\|\hpsi_j-\bpsi_j\|\le \frac{\sqrt{2}\cdot \|\hbSig_{f|y}-\bSigma_{f|y}\|}{\min(|\hlam_{j-1}-\lambda_j|,|\lambda_j-\hlam_{j+1}|)}=O_p(1)\cdot O_p(\omega_{p,T})=O_p(\omega_{p,T}),
\een
where $j=1,\ldots, L$,

Therefore, we complete the proof of Theorem \ref{thm1}.

\end{proof}

\subsection{Proof of Theorem \ref{thm2}}
\begin{proof}
First we write $\hphi=\frac{1}{T-1}\sum_{t=1}^{T-1} y_{t+1}\hf_t$ in terms of the true factors $\bff_t$, i.e.,
\bean
\hphi &=& \frac{\hLam_b}{T-1}\sum_{t=1}^{T-1} y_{t+1}\bx_t \\
        &=& \frac{\hLam_b}{T-1}\sum_{t=1}^{T-1}(\bB\bff_t + \bu_t)y_{t+1},
\eean
where we used the fact that $\hf_t = \hLam_b\bx_t$ as shown in Proposition \ref{prop2.1}.

Using the triangular inequality, we have
\bean
\| \hphi-\bar{\bphi} \|
&\le &  \|\hphi - (\hLam_b\bB)\bar{\bphi}\| + \|  (\hLam_b\bB - \bI_K)\bar{\bphi} \| \\
&=& \|\frac{(\hLam_b\bB)}{T-1}\sum_{t=1}^{T-1}\bff_ty_{t+1} + \frac{\hLam_b}{T-1}\sum_{t=1}^{T-1}\bu_ty_{t+1} - (\hLam_b\bB)\bar{\bphi}\| + \|  (\hLam_b\bB - \bI_K)\bar{\bphi} \| \\
&\leq & \|\frac{(\hLam_b\bB)}{T-1}\sum_{t=1}^{T-1}( y_{t+1}\bff_t - \bar{\bphi})\| + \|\frac{1}{T-1}\hLam_b\sum_{t=1}^{T-1}\bu_ty_{t+1} \| + \| (\hLam_b\bB - \bI_K)\bar{\bphi} \| .
\eean
In the sequel, we will bound three terms on the right hand side of this inequality respectively.

By Lemmas \ref{lem2} and \ref{lem3}, we have
\ben
\|\hLam_b\bB\| \le (\|\hLam_b- \bLam_b\|+\| \bLam_b\|) \cdot \|\bB\|= O_p(1)
\een
and
\ben
\|\hLam_b\bB - \bI_K\| = \|\hLam_b - \bLam_b\| \cdot \|\bB\| = O_p(\omega_{p,T}).
\een
Since $\|\bar{\bphi}\| = O_p(1)$, the third term on the right hand side of the inequality is $O_p(\omega_{p,T})$, i.e.,
\ben
\| (\hLam_b\bB - \bI_K)\bar{\bphi} \|=O_p(\omega_{p,T}).
\een

For the second term, note that $\bu_t$ is independent of $y_{t+1}$, hence $E(\bu_ty_{t+1})=0$. By law of large numbers and $\|\hLam_b \| = O_p(p^{-1/2})$, we have
\ben
\|\frac{1}{T-1}\hLam_b\sum_{t=1}^{T-1}\bu_ty_{t+1} \|=O_p(p^{-1/2})\cdot O_p(T^{-1/2})=O_p(\omega_{p,T}).
\een

It remains to bound the first term $\|\frac{1}{T-1}\sum_{t=1}^{T-1}( y_{t+1}\bff_t - \bar{\bphi})\|$. To this end, we express $\bff_t$ along the basis $\bphi_1,\cdots,\bphi_L$ and their orthogonal hyperplane, namely,
\ben
\bff_t =\sum_{j=1}^L \langle\bff_t,\bphi_j \rangle \bphi_j + \bff_t^{\perp}.
\een
By the orthogonal decomposition of normal distribution, $\langle\bff_t,\bphi_j \rangle$ and $\bff_t^{\perp}$ are independent. In addition, $y_{t+1}$ depends on $\bff_t$ only through  $\bphi_1'\bff_t,\cdots, \bphi_L'\bff_t$, and then, it is conditionally independent of $\bff_t^{\perp}$. It follows from contraction property that $y_{t+1}$ and $\bff_t^{\perp}$ are independent, unconditionally. Thus,
\ben
E(y_{t+1}\bff_t^{\perp}) =E(y_{t+1})E(\bff_t^{\perp}) = \bzero.
\een
Recall that $\bar{\bphi}=\sum_{j=1}^L E((\bphi_j'\bff_t)y_{t+1}) \bphi_j$. Now, we use the triangle inequality to obtain
\bean
&&\|\frac{1}{T-1}\sum_{t=1}^{T-1}( y_{t+1}\bff_t - \bar{\bphi})\| \\
&=& \|\frac{1}{T-1}\sum_{t=1}^{T-1}\bigl[\sum_{j=1}^L (\bphi_j'\bff_t)y_{t+1} \bphi_j + y_{t+1}\bff_t^{\perp} - \sum_{j=1}^L E((\bphi_j'\bff_t)y_{t+1}) \bphi_j\bigr]\|\\
&\le& \sum_{j=1}^L\| \frac{1}{T-1}\sum_{t=1}^{T-1} \bigl[(\bphi_j'\bff_t)y_{t+1} - E((\bphi_j'\bff_t)y_{t+1}) \bigr] \bphi_j\|+ \|\frac{1}{T-1}\sum_{t=1}^{T-1} y_{t+1}\bff_t^{\perp} \|   \\
&\leq&  \sum_{j=1}^L |\frac{1}{T-1}\sum_{t=1}^{T-1} \bigl[(\bphi_j'\bff_t)y_{t+1} - E((\bphi_j'\bff_t)y_{t+1}) \bigr]| + \| \frac{1}{T-1}\sum_{t=1}^{T-1} y_{t+1}\bff_t^{\perp} \|,
\eean
where we used the fact that $\|\bphi_j\|=1$ for any $j$. Note that each term here is $O_p(T^{-1/2})$ by Law of Large Numbers. Then, we have,
\ben
\|\frac{(\hLam_b\bB)}{T-1}(\sum_{t=1}^{T-1} y_{t+1}\bff_t - \bar{\bphi})\|
\le
\|\hLam_b\bB\|\cdot O_p(T^{-\frac{1}2})
=
O_p(\omega_{p,T}).
\een
This concludes the desired bound that $\| \hphi-\bar{\bphi} \| = O_p(\omega_{p,T})$.

Therefore, we complete the proof of Theorem \ref{thm2}.
\end{proof}

\end{document}